\documentclass[a4paper]{amsart}

\usepackage{amsmath, amssymb, latexsym, amsfonts, pigpen, enumitem}
\usepackage[matrix,arrow,curve]{xy}

\numberwithin{equation}{section}
\theoremstyle{plain}
\newtheorem{theorem}[equation]{Theorem}
\newtheorem{corollary}[equation]{Corollary}
\newtheorem{lemma}[equation]{Lemma}
\newtheorem{sublemma}[equation]{Sublemma}
\newtheorem{proposition}[equation]{Proposition}

\theoremstyle{definition}
\newtheorem{definition}[equation]{Definition}

\theoremstyle{remark}
\newtheorem{remark}[equation]{Remark}

\newcommand{\divisor}{\operatorname{div}}

\newcommand{\pour}{\ar@{}[ur]|(0.2){\text{\pigpenfont G}}}
\newcommand{\podr}{\ar@{}[dr]|(0.2){\text{\pigpenfont A}}}

\DeclareMathOperator{\Ker}{Ker}
\DeclareMathOperator{\Image}{Im}

\DeclareMathOperator{\Coker}{Coker}
\DeclareMathOperator{\Supp}{Supp}
\DeclareMathOperator{\codim}{codim}
\DeclareMathOperator{\Ann}{Ann}

\newcommand{\cal}{\mathcal}

\newcommand{\bGm}{{\mathbb G_m}}

\newcommand{\pri}[1]{{#1^\prime}}

\newcommand{\ppri}[1]{{#1^{\prime\prime}}}

\newcommand{\prl}{\mathbb{P}^1}
\newcommand{\affl}{\mathbb{A}^1}
\newcommand{\aff}{{\mathbb{A}}}

\newcommand{\bZ}{\mathbb Z}

\newcommand{\bX}{\mathcal{X}}
\newcommand{\bXp}{\pri{\mathcal{X}}}
\newcommand{\bXpp}{\ppri{\mathcal{X}}}
\newcommand{\ovbX}{ \overline{ \mathcal X } }
\newcommand{\ovbXp}{ \overline{ \pri{\mathcal X} } }
\newcommand{\ovbXpp}{ \overline{ \ppri{\mathcal X} } }
\newcommand{\ovXp}{\overline{X^\prime}}
\newcommand{\ovX}{\overline X}
\newcommand{\ovpi}{\overline\pi}
\newcommand{\ovvarpi}{\overline\varpi}
\newcommand{\bZp}{\cal Z^\prime}
\newcommand{\bZpp}{\ppri{\mathcal{Z}}}

\newcommand{\tXp}{\tilde{X}^\prime}
\newcommand{\ovtXp}{\overline{X}^\prime}
\newcommand{\ovtX}{\overline{X}}
\newcommand{\tZp}{\tilde{X}^\prime}
\newcommand{\ovtpi}{\overline{\tilde \pi}}

\begin{document}

\title{Strictly homotopy invariance of Nisnevich sheaves with GW-transfers}
\author{Andrei Druzhinin}
\address{Chebyshev Laboratory, St. Petersburg State University, 14th Line V.O., 29B, Saint Petersburg 199178 Russia}
\email{andrei.druzh@gmail.com}
\thanks{Research is supported by the Russian Science Foundation grant 14-21-00035}
%\subjclass[2010]{14F42, 19G38, 19E20, 19G12}

\begin{abstract}
%We prove that for any homotopy invariant presheave $\cal F$ with GW-transfers (Witt-transfers) on the category of smooth varieties over a prefect field $k$, $char\,k \neq 2$,
%%(according to definition from \ref{})
%the associated Nisnevish sheave $\widetilde{\cal F}_{Nis}$ is strictly homotopy invariant, i.e. $$H^i_{Nis}(\affl\times X,\widetilde{\cal F}_{Nis})\simeq H^i_{Nis}(X,\widetilde{\cal F}_{Nis})$$ for any $X\in Sm_k$.
%This theorem is necessary 
%in the construction of the triangulated category of GW-motives $\mathbf{DM}^{GW}(k)$ and Witt-motives $\mathbf{DM}^W(k)$ by the Voevodsky-Suslin method originally used for the construction of the category of motives $\mathbf{DM}(k)$. % similar method as in the construction of Voevodsky motives $DM(k)$.

%We prove that for any homotopy invariant presheave $\cal F$ with GW-transfers (Witt-transfers) on the category of smooth varieties over a prefect field $k$, $char\,k \neq 2$,
%(according to definition from \ref{})
%the associated Nisnevish sheave $\widetilde{\cal F}_{Nis}$ is 
The strictly homotopy invariance of the associated Nisnevish sheave $\widetilde{\cal F}_{Nis}$ of a 
homotopy invariant presheave $\cal F$ with GW-transfers (or Witt-transfers) on the category of smooth varieties over a prefect field $k$, $char\,k \neq 2$,
is proved, 
i.e. the isomorphism $$H^i_{Nis}(\affl\times X,\widetilde{\cal F}_{Nis})\simeq H^i_{Nis}(X,\widetilde{\cal F}_{Nis})$$ for any $X\in Sm_k$ is obtained.
This theorem is necessary 
for the construction of the triangulated category of GW-motives $\mathbf{DM}^{GW}(k)$ and Witt-motives $\mathbf{DM}^W(k)$ by the Voevodsky-Suslin method originally used for the construction of the category of motives $\mathbf{DM}(k)$. % similar method as in the construction of Voevodsky motives $\mathbf{DM}(k)$.

In particular, the result of the article gives the direct prove of the  strictly homotopy invariance of 
the Nisnevich sheaves associated to hermitian K-theory and Witt-groups
(without using of the representability of these cohomology theories in the motivic homotopy category $\mathbf H_{\affl}(k)$ proved by Hornbostel \cite{Horn_ReprKOWitt});
and on other side
the strictly homotopy invariance theorem proved here and the representability criteria proved in \cite{Horn_ReprKOWitt} 
implies that cohomologies $H^i_{nis}(-,\widetilde{\mathcal F}_{nis})$ of the associated sheaf of a homotopy invariant presheave with GW-(Witt-)transfers $\mathcal F$ are representable in $\mathbf H_{\affl}(k)$.
%As the immediate consequence of the strictly homotopy invariance theorem above and the representability criteria proved in 
\end{abstract}

\maketitle

\section{Introduction.}

In this article we prove that the Nisnevich sheave associated to a homotopy invariant presehave with GW-transfers
is strictly homotopy invariant. 
This result is necessarily
for the construction of the category GW-motives $\mathbf{DM}^{GW}(k)$ 
by the Voevodsky-Suslin-method originally used for construction of the category of motives $\mathbf{DM}(k)$ (see \cite{VSF_CTMht_Ctpretr}, \cite{VSF_CTMht_DM}, \cite{SV_Bloch-Kato}, \cite{MVW_LectMotCohom}). 
Saying Voevodsky-Suslin-method we imply that we start with some additive category of correspondences (GW-correspondences $GWCor_k$ of Witt-correspondences $WCor_k$),
and define $\mathbf{DM}^{GW}$ as $\bGm^{\wedge 1}$-stabilisation of the category of effective GW-motives $\mathbf{DM}^{GW}_\mathrm{eff}$ and define 
$\mathbf{DM}^{GW}_\mathrm{eff}$ %($\mathbf{DM}^W_\mathrm{eff}$) 
as the full subcategory in derived category of the category of sheaves with GW-transfers,
spanned by motivic complexes, i.e. complexes with homotopy invariant sheaf cohomology (and similarly for Witt-motives).

By definition Nisnevich sheaves with GW-transfers (Witt-transfers) are presheaves with GW-transfers (Witt-transfers) that are sheaves and 
presheaves with GW-transfers (Witt-transfers) are just additive presheaves on the category of GW-correspondences $GWCor_k$,
To define in short the category of GW-correspondences (Witt-correspondences) let's say that for affine schemes $X,Y$ the morphism group$GWCor(X,Y)$ is the Grothendieck-Witt-group of the quadratic spaces 
$(P,q)$, where $P\in k[Y\times X]-mod$, that are finitely generated projective over $k[X]$,
and $q\colon P\simeq Hom_{k[X]}(P,k[X])$ is $k[Y\times X]$-linear isomorphism.
The category Witt-correspondences is defined in the same way using Witt-groups.
% category with duality, i.e. as group of isomorphism classes of quadratic spaces factorised by metabolic spaces.
%In the case of affine schemes
%q%uadratic spaces that defines correspondences between schemes $X,Y$ are pairs $(P,q)$,
%where $P\in k[Y\times X]-mod$, that are finitely generated projective over $k[X]$,
%and $q\colon P\simeq Hom_{k[X]}(P,k[X])$ is $k[Y\times X]$-linear isomorphism,
%and then classes of such pairs $[(P,q)]$ defines GW-correpondences between affine varieties $X$ and $Y$.
And in general case we replace $k[Y\times X]$-module $P$ by coherent sheave on $X\times Y$ 
 used in the definition of the category of $K_0$ correspondences studied by Walker in \cite{W_MotComK} %and Suslin in \cite{Suslin-GraysonSpectralSeq} 
and $K$-correspondences studied by Garkusha and Panin in \cite{GarPan-Kmot}.
Namely we consider the coherent sheave $P$ on $X\times Y$ that support is finite over $X$ and that direct image on $X$ is locally free coherent sheave of finite rank.

According to Grothendieck's idea any category of motives 
plays role of the 'universal' cohomology theory for some class of cohomology theories, 
that means that all cohomology theories of this class have canonical lift to functors defined on the category of motives,
and the category of motives provides 'geometrical' instruments for computations of this theories.
For triangulated categories of motives this means 
that this cohomology theories are represented and representing objects generates this triangulated category.
In this sense the categories of GW-motives (or Witt-motives) are natural containers %relates to 
for the homotopy invariant Nisnevich excisive cohomology theories equipped with GW-transfers (or Witt-transfers), and the main examples of such theories are the higher hermitian K-theory $GW^i(-)$ (or the derived Witt-groups $W^i(-)$) and cohomologies of the associated Nisnevich sheaves. 
So this categories gives the geometrical framework for these cohomology theories, and in particular, this allows to apply the Voevodsky method of the proof of the Gersten conjecture (used originally for pretheories with transfers defined by the category $Cor$) to get the alternative proof of the Gersten conjecture for hermitian K-theory and Witt-groups.
The categories $\mathbf{DM}^{GW}(k)$ (and $\mathbf{DM}^W$) can be useful for a construction of spectral sequences converging to hermitian K-theory like that the categories of $K_0$-motives and $K$-motives ware used for the Grayson motivic spectral sequences \cite{Suslin-GraysonSpectralSeq}, \cite{GarPan-MotSpSeq}, \cite{GarPan-Kmot}.
%d 
%Can be used for the Gersten conjecture for presheaves with GW-transfers

Let's note also that 
it follows from the results of \cite{FB_EffSpMotCT}  and \cite{GG_RecRatStMot} that 
the category $\mathbf{DM}^{GW}(k)$ is rationally isomorphic to the stable motivic category $\mathbf{SH}(k)_{\mathbb Q}$.
%As noted in \cite{FB_EffSpMotCT} after the inverting of $p = \chark k$ in the coefficient ring (i.e. with $\mathbb Z[1/p]$ coefficients) this theory is equivalent to the generalised motivic cohomology defined by Calmes and Fasel in \cite{CF_FinChWittCor} via Chow-Witt groups and the category of Milnor-Witt-motives, so it follows from the result of \cite{GG_RecRatStMot} by Garkusha that $\mathbf{DM}^{GW}(k)$ is rationally equivalent to the stable motivic homotopy category $\mathbf{SH}(k)_{\mathbb Q}$, and 
%So computation of the GW-motives of smooth schemes presented here gives fibrant replacements in $\mathbf{SH}(k)_{\mathbb Q}$. 
The category $\mathbf{DM}^{W}(k)$ is hypothetically equivalent to the category of Witt-motives
constructed by Ananievsky, Levine, Panin in \cite{ALP_WittSh-etsinvert} via the category modules over the Witt-ring sheaf,
and so $\mathbf{DM}^{W}(k)$ is rationally equivalent to the minus part $\mathbf{SH}^-(k)_{\mathbb Q}$ of the stable motivic homotopy category.

Since as mentioned above the category of effective GW-motives (Witt-motives) $\mathbf{DM}^{GW}_{\mathrm{eff}}$ $(\mathbf{DM}^W_{\mathrm{eff}}$) should satisfy the universal property for the class of homotopy invariant Nisnevich excisive cohomology theories with GW-transfers (Witt-transfers),
it is natural to define the categories of GW-motives and Witt-motive as localisation of the derived category of the category of Nisnevich sheaves with GW-transfers (Witt-transfers) in respect to $\affl$-equivalences $L_{\affl}\colon \mathbf{D}(ShNisGWtr)\to \mathbf{DM}^{GW}_{\mathrm{eff}}(k).$
%(or as localisation of the derived category of presheaves with GW-transfers (Witt-transfers) in respect to Nisnevich equivalence and $\affl$-equivalences)
%of $D^-(ShNisGWtr)$.
The important 
advantage of the Voevodsky-Suslin method is that in the case of a perfect base filed
this method provides the computation of the right adjoint functor $R_{\affl}\colon \mathbf{DM}^{GW}_{\mathrm{eff}}(k)\to \mathbf{D}(ShNisGWtr)$
as the full embedding by the subcategory of motivic complexes defined above. 
Then 
the localisation functor $L_{\affl}$ is equal to internal Hom-functor $\mathcal Hom(\Delta^\bullet,-)$
represented by the complex corresponding to infinite affine simplex $\Delta^\bullet$. 
Thus this computation of the category $\mathbf{DM}^{GW}_{\mathrm{eff}}$ ($\mathbf{DM}^{W}_{\mathrm{eff}}$) and functors $L_{\affl}$ and $R_{\affl}$ gives an instrument for computation of $Hom$-groups in the category of effective GW-motives (Witt-motives) and in particular it can be useful for the computations of the mentioned cohomology theories.

The critical point in the computation of the the functor $R_{\affl}$
according to the Voevodsky-Suslin method is the following theorem,
that is the main result of the article:
\begin{theorem}\label{th:intr:StrHomInv}
Nisnevich sheafification $\widetilde{\cal F}_{nis}$ of any homotopy invariant presheaf $\cal F$ with GW-transfers is strictly homotopy invariant.
\end{theorem}
Similar to the original case of $Cor$-correspondences used in the construction of the category $\mathbf{DM}^-(k)$ %and other cases
the proof of the theorem above is based on the computation of 
cohomology groups on relative affine line $\affl_U$ over a local base
of the Nisnevich sheaf $\widetilde{\cal F}_{nis}$
associated with homotopy invariant presheave with GW-transfers
(lemma \ref{lm:strinv-locbase}): 
\begin{equation}\label{eq:cohA1}
\widetilde{\cal F}_{nis}(\affl_U)\simeq \cal F(\affl_U), H_{nis}^i(\affl_U,\widetilde{\cal F}_{nis})=0,\text{ for }i>0.
\end{equation}
This equalities essentially uses transfers defined by considered category of correspondences $GWCor_k$, 
and proof is based on 
the explicit construction of some GW-correspondences (Witt-correspondences) between etale coverings of open subschemes in relative affine.
So the proof differs for different categories of correspondences,
and the main innovative ingredients in the proof are 
geometrical constructions that allows to control 'orientation' of correspondences and to define the required quadratic forms,
and this is the most essential novelty of the work.

The role of GW-correspondences for equality \ref{eq:cohA1} can be simply explained if we %can 
think about transfers on presheaves
as representations of some 'ring'
corresponding to the category of correspondences\footnote{here we speak about the informal analogy rather then a strict mathematical notion, though indeed formally we can think, for example, about the corresponding ring spectra in the motivic homotopy.}. % category any category of correspondences} 
Then using analogy between
(homotopy invariant) (pre-)sheaves with transfers on the category of smooth schemes and
coherent (pre-)sheaves on some scheme,
we see that equality \eqref{eq:cohA1} is analogously to the fact
that any coherent shave on affine scheme is just a module over the function ring 
and cohomologies of coherent sheaves on affine schemes are zero.
The mentioned fact about coherent sheaves relates to the existence of unit decomposition in the ring of functions.
In the same sense equality \eqref{eq:cohA1} relates to 
the existence of a '$\affl$-homotopy decomposition' in the category of correspondences of the identity on the affine line, 
i.e. a lift along the Nisnevich covering $\cal U\to \affl_U$.
in the category of correspondences up to $\affl$-homotopy
of the identity morphism $id_{\affl_U}$ 
So presented here constructions of correspondences shortly speaking gives such decomposition of unit in the category of GW-(Witt-)correspondences up to $\affl$-homotopy.

Formally proof of equality \ref{eq:cohA1} are based
on the following excision and injectivity theorems:
\begin{theorem}[etale excision, theorem \ref{th:et-ex}]\label{th:intr:EtEx}
For a homotopy invariant presheave with GW-transfers $\cal F$,
etale morphism of essentially smooth local schemes $\pi\colon  V^\prime\to V$,
closed subscheme $Z\subset V$ of codimension 1,
such that  
$\pi$ induces isomorphism between 
$Z$ and its preimage  $Z^\prime=\pi^{-1}(Z)$,
$\pi$ induces the isomorphism  
   $$\pi^*\colon \frac{\cal F(V-Z)}{\cal F(V)}  
              \stackrel{\sim}{\to}  
                     \frac{\cal F(V^\prime-Z^{\prime})}{\cal F(V^\prime)},$$
\end{theorem}
\begin{theorem}[Zariski excision on relative affine line, theorem \ref{th:afflUzar-ex}]\label{th:intr:ZarEx}
For be homotopy invariant sheave with GW-transfers $\cal F$ 
Zariski open subvariety $V\subset\affl_U\colon V\supset 0_U$
for essential smooth local scheme $U$,
restriction homomorphism induce isomorphism 
    $$\frac{\cal F(\affl_U-0_U)}{\cal F(\affl_U)}\simeq \frac{\cal F(V-0_U)}{\cal F(V)}$$
%  ($i$ denotes embedding of $V$ into  $\affl_U$).  
\end{theorem}
\begin{theorem}(theorem \ref{th:affZar-inj})\label{th:intr:ZarAffInj}
Let 
  $\cal F$ is be homotopy invariant sheave with GW-transfers over field $k$
  and $K$ be geometric extension $K/k$ (i.e. field of functions of some variety).
Then 
  for any Zariski open subschemes $U\subset V\subset\affl_K$
  restriction homomorphism   
    $$i^*\colon {\cal F(V)}\to {\cal F(U)}$$
  is an injective, 
  where $i\colon U\hookrightarrow V$ denotes open immersion.  
\end{theorem}
\begin{theorem}[see \cite{ChepInjLocHIiWtrPreSh}  for the case of Witt-correspondences, and theorem \ref{th:locZar-inj} for the case of GW-correspondences]\label{th:intr:InjTh}
For any essential smooth local scheme $U$,
and closed subscheme $Z\subset U$,
restriction homomorphism
    $$i^*\colon {\cal F(V)}\to {\cal F(U)}$$
is injective.
\end{theorem}
To prove the excision theorem above %\ref{} and \ref{}
we give an explicit construction of GW-correspondences in the category of pairs
$(V,V-Z)\to(V^\prime, V^\prime-Z)$, that are inverse up to $\affl$-homotopy to
embedding morphisms
$(V,V-Z)\to(V^\prime, V^\prime-Z)$
where $Z$ is considered closed subscheme and $V$ is corresponding neighbourhood of $Z$ (Nisnevich or Zarisky).
To prove the injectivity theorems we give the construction of a left inverse in the category of GW-(Witt)-correspondences between pairs to the mentioned embeddings of open subschemes.

%Let's explain now in short what are mentioned above GW-transfers and  GW-correspondences.
%Presheaves with GW-transfers are additive presheaves on the category of GW-correspondences $GWCor_k$.
%Shortly speaking GW-correspondences are defined as Witt-groups of some category with duality, i.e. as group of isomorphism classes of quadratic spaces factorised by metabolic spaces.
%In the case of affine schemes
%quadratic spaces that defines correspondences between schemes $X,Y$ are pairs $(P,q)$,
%where $P\in k[Y\times X]-mod$, that are finitely generated projective over $k[X]$,
%and $q\colon P\simeq Hom_{k[X]}(P,k[X])$ is $k[Y\times X]$-linear isomorphism,
%and then classes of such pairs $[(P,q)]$ defines GW-correpondences between affine varieties $X$ and $Y$.
%And in general case we replace $k[Y\times X]$-module $P$ by coherent sheave on $X\times Y$ that support is finite over $X$ and that direct image on $X$ is locally free coherent sheave of finite rank.

%Excision theorem above %\ref{} and \ref{}
%are proved by explicit construction of GW-correspondences in the category of pairs
%$(V,V-Z)\to(V^\prime, V^\prime-Z)$, that are inverse up to $\affl$-homotopy to
%embedding morphisms
%$(V,V-Z)\to(V^\prime, V^\prime-Z)$
%where $Z$ is considered closed subscheme and $V$ is corresponding neighbourhood of $Z$ (Nisnevich or Zarisky).
%So the main ingredients in the proofs are 
%geometrical constructions that allows to control 'orientation' of correspondences and define required quadratic form,
%and this is the most essential novelty of this work.

\subsection{Overview of the text:}
In the section \ref{sect:GWW-cor} we present the definitons of the categories of GW-(Witt-)correspondences, prove the basic elementary properties and give the construction that produce GW-(Witt-)correspondences from a relative curve with trivialisation of the relative canonical class and a (good) regular function of the curve.

In the sections \ref{sect:QuiltrCompact} and \ref{sect:FunctionsConstr} we summarise geometric constructions used in the construction of GW-correspondences in the proofs of the excision and injectivity theorems.

In the section \ref{sect:ZarEx}
the Zariski excision isomorphism (theorem \ref{th:intr:ZarEx}) on the relative affine line over a local base is proved.
In the section \ref{sect:EtEx}
the etale excision isomorphism (theorem \ref{th:intr:EtEx}) on the relative affine line over a local base is proved.
In the section \ref{sect:Injth}
the injectivity theorems on a local essential smooth scheme (theorem \ref{th:intr:InjTh}) and on the affine line (theorem \ref{th:intr:ZarAffInj}) are proved.

The section we give \ref{sect:StrHomInv} the main result of the article (theorem \ref{th:intr:StrHomInv}). 
%\ref{sect:StrHomInv} contains the formulations and proofs of the main results of the article.

\subsection{Acknowledgements}
Acknowledgement to I.~Panin who encouraged me to work on this project, for helpful discussions.

\subsection{Notation}
%Further in text we use the following notations: 
All schemes are a separated noetherian schemes of finite type over the base, 
and $Sm_k$ denotes the category of smooth schemes over field $k$.
We denote $Coh(X)=Coh_X$ the category of coherent sheaves on a scheme $X$,
for any scheme $X$ (not only affine) to shortify denotations. 
For any $P\in Coh(X)$ we denote by $\Supp P$ the closed subscheme in $X$ defined by the sheaf of ideals $\mathcal I(U)=\Ann P\big|_{U}\subset k[U]$ and we denote by $\Supp_{red} P$ the reduced subscheme of $\Supp P$.

We write $k[X]$ for the ring of regular (global defined) functions $X\to \affl$, i.e. $k[X] = \Gamma(X,\mathcal O(X))$. 
We denote by $Z(f)$ vanish locus of $f$, for any regular function $f$ on scheme $X$, 
and by $Z_{red}(f)$ its reduced subscheme. 
Similarly,
for a section $s\in \Gamma(X,\cal L)$ of some line bundle on scheme $X$ 
we denote by $Z(s)$ the closed subscheme defined by the ideal sheaf $\{f\colon div\,f\geq div\,s\}$ 
(which is equivalent to the image of homomorphism $\cal L_{-1}\xrightarrow{s} \cal O(X)$).
For an effective divisor $D$ in variety $X$ denote by $S(D)$ the closed subscheme $Z(s)$, 
where $s\in \Gamma(X,\cal L(D)),\,div\,s=D.$

%We call by varieties a reduced separated noetherian schemes of finite type over the base.
%The base field in prefect, unless otherwise specified.
%%and we call by curve over base scheme $S$ a variety $X$ over $S$ such that dimension of all fibres of canonical morphism $X/to S$ is one.

\section{GW-correspondences}\label{sect:GWW-cor}

\subsection{categories with duality $(\mathcal P(Y\to X),D_X)$.}

\begin{definition}
For a morphism of schemes $p\colon Y\to X$,
let $Coh_{fin}(p)$ (or $Coh_{fin}(Y_X)$) denotes
the full subcategory of the category of coherent sheaves on $S$ 
  spanned by sheaves $\cal F$ such that $\Supp \cal F$ is finite over $X$;
and let  
$\mathcal P(p)$ (or $\mathcal P(Y_X)$) 
denotes the full subcategory of $Coh_{fin}(Y)$ 
  spanned by sheaves $\cal F$ such that $p_*(\cal F)$ is locally free sheave on $X$. % (and automatically it has finite rank).

%In particular, 
For two schemes %smooth varieties 
$X$ and $Y$ over a base scheme $S$
we denote $$Coh_{fin}^S(X,Y) = Coh_{fin}(X\times_S Y\to X),\;\mathcal P^S(X,Y) = \mathcal P(X\times_S Y\to X).$$ 
\end{definition}
\begin{remark}
In the case of affine schemes $Y$, $X$,
$\cal P(Y\to X)$ is equivalent to the full subcategory in the 
category of $k[Y]$-modules consisting of modules that are finitely generated and projective over $k[X]$.
\end{remark}
%\begin{lemma}\label{finSuppQC}
%For 
%  any affine $X$, $Y$ and 
%  any quadratic space $\Phi$ in $\ProjC(X,Y)$ 
%$\Supp \Phi$ is finite over $X$.
%\end{lemma}
%{\em Proof of the lemma.}
%Let $\Phi = (P,q_P)$.
%By definition 
%  $$\Supp \Phi = Spec k[X\times Y]/ Ann\,P.$$
%Let      
%  $m_1,\dots,m_n$ be finite set of generators of $P$ over $k[X]$
%then
%  $$Ann\,P=\bigcap\limits_{i=1}^n Ann\,m_i$$
%and
%  the homomorphism of $k[X]$-modules
%  $$(1,\dots 1)\colon R/Ann\,P\to \bigoplus\limits_{i=1}^n R/Ann\,m_i$$
%  is embedding.
%The composition of 
%  this homomorphism with 
%  embeddings   $R/Ann\,m_i\hookrightarrow P$ 
%gives us 
%  the embedding of $R/Ann\,P$ into $\bigoplus\limits_{i=1}^n P$.
%So $R/Ann\,P$ is isomorphic to the 
%  submodule of the finitely generated $k[X]$-module
%and since $k[X]$ is Noetherian
%  $R/Ann\,P$ is finitely generated over $k[X]$. 
%{\em Lemma is proved.}

The internal hom-functor $D_X=\mathcal Hom( -,\mathcal O(X) )$ on the category of coherent sheaves on $X$ can be naturally lifted to a functor $$D_X\colon Coh_{fin}(Y_X)^{op}\to Coh_{fin}(Y_X)^{op},$$ for any morphism of schemes $Y\to X$.

Indeed,
firstly let's note that we can define the required functor $D_X$ locally along $X$,
i.e. it is enough to define the functor $D_X$ in a natural way for affine schemes $X$.
Next let's note that if $X$ is affine and $Y\to X$ is finite morphism,
then $Coh_{fin}(Y\to X)\simeq Coh(Y)\simeq k[Y]-mod$, 
and the required functor $D_X$ is equivalent to the functor $Hom(-,k[X])\colon k[Y]-mod\to k[Y]-mod$. %can be defined as 
In general case of a morphism $Y\to X$
we have $$Coh_{fin}(Y\to X)=\varinjlim\limits_{Z} Coh(Z\to X),$$ where $Z$ ranges over the set of closed subschemes in $Y$ finite over $X$;
and hence the functor $D_X$ on $Coh_{fin}(Y\to X)$ can be defined as a direct limit of functors $D_X$ defined on $Coh_{fin}(Z\to X)=Coh(Z)$ that are just defined by the above.

Next let's note that since the functor $Coh_{fin}(Y\to X)\to Coh(X)$ is conservative
and $D_X$ on $Coh(X)$ defines the duality on the subcategory category $\mathcal P(X)$ spanned by locally free coherent sheaves,
it follows that the functor $D_X$ on $Coh_{fin}(X)$ defines the duality on the category $\mathcal P(Y\to X)$.
In addition
for any morphisms of schemes $X_3\to X_2\to X_1$ 
the tensor product of modules (coherent sheaves)
defines a functor of categories with duality
\begin{equation}\label{eq:comp}
-\circ -\colon  (\mathcal P^{X_3}_{X_2},D_{X_2}) \times (\mathcal P^{X_2}_{X_1},D_{X_1}) \to (\mathcal P^{X_3}_{X_1},D_{X_1}), 
\end{equation} %\eta\colon D_X(-\circ -)\simeq D_S(-)\circ D_X(-)$$
%for any morphisms of schemes $X_3\to X_2\to X_1$ and naturally in $X_3$,$X_2$,$X_1$ 
which is natural along $X_3$, $X_2$, $X_1$ and 
satisfies the associativity in that sense that for any three morphisms $X_4\to X_3\to X_2\to X_1$
and $P_3\in \mathcal P(X_4,X_3)$, $P_2\in\mathcal P(X_3,X_2)$, $P_2\in\mathcal P(X_2,X_1)$,
there is a natural isomorphism $\xi\colon P_3\circ (P_2\circ P_1)\simeq (P_3\circ P_2)\circ P_1$, 
such that
$(\xi_{1,2,3}\circ -)\circ (\xi_{1,23,4})\circ (- \circ \xi_{2,3,4}) = \xi_{12,3,4}\circ \xi_{1,2,34}$,
and
%commutes with $\eta$
$(\eta_{1,2}\circ D(-)) \circ \eta_{12,3}\circ D(\xi) = D(\xi)^{-1}\circ (D(-)\circ\eta_{2,3}) \circ \eta_{1,23}$,
where $\eta_{i,j}\colon D_{X_i}(-\circ -)\simeq D_{X_j}(-)\circ D_{X_i}(-)$ ($i,j=1,\dots, 4$) denotes the structure morphism of the functor of categories with duality.

\begin{remark}
%1)
%For any morphism $Y\to X$,
%\begin{equation}\label{eq:GWCorfininjlim}
%Coh_{fin}(Y_X) = \varinjlim\limits_{Z} Coh(Z),\quad \cal P(Y_X) = \varinjlim\limits_{Z} \cal P(Z),\end{equation}
%where $Z$ goes throw all subschemes of $Y$ that are finite over $X$.
%
%2)
The functor $D_X$ is represented by $p^!(\cal O(X))$, i.e.
$D_X(\cal F) = \cal Hom(\cal F,p^!(\cal O(X))$
where $\mathcal Hom$ denotes internal homomorphism functor in $Coh(Y)$.

\end{remark}

\subsection{Categories $GWCor$, $WCor$.}

\begin{definition}\label{def:GWCor}
The \emph{category $QCor_S$} is the category  
with objects being smooth schemes over $S$, 
morphism groups being defined as 
$$QCor_S(X,Y) = Q(\mathcal P^S(X,Y),D_X)$$
where the symbol $Q$ denotes the set of isomorphism classes of quadratic spaces in the category with duality %$(\mathcal P^S(X,Y),D_X)$ Grothendieck-Witt-group of the exact category with duality,
%i.e. the group completion of the groupoid (up to direct sums) of non-degenerate quadratic spaces $(P,q)\colon$ $P\in \mathcal P^S(X,Y)$, $q\colon P\simeq D_X(P)$,
the composition is induced by the functor \eqref{eq:comp},
and identity morphism $$Id_X=[(\mathcal O(\Delta),1)],$$ 
where $\Delta$ denotes diagonal in $X\times_S X$.

%The \emph{category $GWCor_S$} is the additive category  
%with objects smooth schemes over $S$, 
%morphism groups are defined as 
%$$GWCor_S(X,Y) = GW(\mathcal P^S(X,Y),D_X)$$
%where $GW$ is Grothendieck-Witt-group of the exact category with duality,
%i.e. the group completion of the groupoid (up to direct sums) of non-degenerate quadratic spaces $(P,q)\colon$ $P\in \mathcal P^S(X,Y)$, $q\colon P\simeq D_X(P)$,
%the composition is induced by the functor \eqref{eq:comp},
%and identity morphism $$Id_X=[(\mathcal O(\Delta),1)],$$ 
%where $\Delta$ denotes diagonal in $X\times_S X$.
%

The \emph{categories $GWCor_S$ and $WCor_S$} are 
the additive categories with the same objects and 
such that 
$$GWCor(X,Y)= GW(\mathcal P^S(X,Y),D_X), 
WCor(X,Y)= W(\mathcal P^S(X,Y),D_X),$$
where $GW$ denotes the Grothendieck-Witt-group of the exact category with duality,
i.e. the group completion of the groupoid (up to direct sums) of non-degenerate quadratic spaces $(P,q)\colon$ $P\in \mathcal P^S(X,Y)$, $q\colon P\simeq D_X(P)$,
and $W$ are Witt group of the exact category with duality (see Balmer \cite{Bal_DerWitt}). 

%The \emph{category $GWCor_S$} is the additivisation of the category $QCor_S$. 
%%
%The \emph{category $WCor_S$} is the factor-category of $GWCor_S$ such that classes of metabolic spaces defines the zero morphism. 

%The \emph{category $GWCor_S$} is the additivisation of the category $QCor_S$, or equivalently it is the additive category
%with the same objects and $$GWCor(X,Y)= GW(\mathcal P^S(X,Y),D_X),$$
%where $GW$ denotes the Grothendieck-Witt-group of the exact category with duality,
%i.e. the group completion of the groupoid (up to direct sums) of non-degenerate quadratic spaces $(P,q)\colon$ $P\in \mathcal P^S(X,Y)$, $q\colon P\simeq D_X(P)$.
%
%The \emph{category $WCor_S$} is the additive category  
%defined in a same way using the Witt-groups of an exact category with duality instead of Grothendieck-Witt groups (see Balmer, \cite{Bal_DerWitt}), 
%i.e. $WCor_S$ is factor-category of $GWCor_S$ such that classes of metabolic spaces defines the zero morphism. 
\end{definition}  
\begin{remark}
Equivalently to the definition above we can say that 
the \emph{category $GWCor_S$} is the additivisation of the category $QCor_S$,
and the \emph{category $WCor_S$} is the factor-category of $GWCor_S$ such that classes of metabolic spaces defines the zero morphism. 
%Equivalently to the definition above we can say that 
%$$GWCor(X,Y)= GW(\mathcal P^S(X,Y),D_X), 
%WCor(X,Y)= W(\mathcal P^S(X,Y),D_X)$$
%where $GW$ denotes the Grothendieck-Witt-group of the exact category with duality,
%i.e. the group completion of the groupoid (up to direct sums) of non-degenerate quadratic spaces $(P,q)\colon$ $P\in \mathcal P^S(X,Y)$, $q\colon P\simeq D_X(P)$,
%and $W$ are Witt group of the exact category with duality, see Balmer, \cite{Bal_DerWitt}. 
\end{remark}
\begin{definition}\label{def:funSmGWW}
Let's define a functor $Sm_S\to GWCor_S$,
$$f\in Mor_{Sm_k}(X,Y) \mapsto [(\cal O(\Gamma_f), 1)],$$
where $\Gamma_f$ denotes graph of morphism $f$, that is closed subscheme in $Y\times X$ isomorphic to $X$,
and $1$ denotes unit quadratic form on a free coherent sheave of a rank one. % (we imply isomorphism . % isomorphism $
The composition with factorisation $GWCor_S\to WCor_S$ gives us the functor $Sm_S\to WCor_S$.
\end{definition}
\begin{remark}
For any $\Phi\in GWCor$ and regular maps $f\colon X^\prime \to X$ and $g\colon Y\to Y^\prime$,
$$\Phi\circ f = (id_Y\times f)^*(\Phi),\;g\circ \Phi = (g\times id_X)_*(\Phi),$$
where $(id_Y\times f)^*$ denotes inverse image along morphism $id_Y\times f\colon Y\times X^\prime\to Y\times X$,
and $(g\times id_X)_*$ denotes direct image along morphism $g\times id_X\colon Y\times X^\prime\to Y\times X$.
\end{remark}  
The following definitions and lemmas can be given in the same manner for $GW$-correspondences and $Witt$-correspondences.
 
\begin{definition}
%For a base field $k$,
A \emph{presheave on $Sm_S$} is an additive functor $F\colon Sm_S\to Ab$;     
a \emph{presheave with GW-transfers} over a base $S$ is an additive functor $F\colon GWCor_S\to Ab$.     
%\end{definition}
%\begin{definition}

A presheave $\mathcal F$ on $Sm_S$ is called \emph{homotopy invariant} is the natural homomorphism $\cal F(X)\simeq \cal F(\affl\times X)$ is an isomorphism for any $X\in Sm_S$;
a presheave $\mathcal F$ with GW-transfers is \emph{homotopy invariant} if it is homotopy invariant as a presheave on $Sm_S$ via the functor from definition \ref{def:funSmGWW}.     
\end{definition}

\begin{lemma}\label{lm:GWCor:locbase.}
Suppose $S$ is a scheme and $s\in S$ is a point;
then there is an embedding functor $GWCor_{S_s}\to \text{\rm pro-}GWCor_S$ (and $WCor_{S_s}\to \text{\rm{pro-}}WCor_S$).
So consequently 
any presheave with GW-transfers over $S$ defines in canonical way a presheave with GW-transfers over $S_s$.
%For any point in a base scheme $s\in S$,
%the category $GWCor_{S_s}$ over the local scheme $S_s$ is 
%the subcategory of the category spanned by pro-objects of $GWCor_S$,
%and consequently any presheave with GW-transfers over $S$ defines in canonical way a presheave with GW-transfers over $S_s$.
\end{lemma}\begin{proof}
We omit the full prove to shortify text. 
Let's note only that 
the claim follows form the following points: 
1) any scheme of finite type over $S_s$ is porjective limit of schemes over Zariski neighbourhoods $s\in U_i\in S$;
2) any quadratic space can be defined by finite the set of data (regular functions and equations); 
the isomorphism of quadratic spaces and the property of a quadratic space to be metabolic can be defined by finite set of data (regular functions and equalities).
% which is well defined over some Zariski neighbourhood of point $s\in S$.
Note that for this statement it is essential that schemes considered in definition \ref{def:GWCor} %(i.e. smooth schemes) 
are schemes of finite type over the base. %$K/k$ finite type schemes. 
%, so ,
%and property of quadratic space metabolic spaces
\end{proof}
\begin{definition}
An algebra $R/k$ is called \emph{geometric extension} of the base filed $k$ if
$R$ is isomorphic to the local ring $k[X_x]$ for some smooth variety $X$ over $k$ and point $x\in X$.
%We call by geometric extension of the base filed $k$ a filed $K/k$ that is isomorphic to fields of functions $K=k(X)$ for some smooth variety $X$ over $k$.
%, i.e. fields $ 
\end{definition}
\begin{corollary}\label{cor:GWCor:getombase}
A presheave with GW-transfers over $k$ defines in canonical way a presheave with GW-transfers over $R$
for any geometric extension $R/k$. 
\end{corollary}

\begin{definition}\label{GWCorpair}
We define additive category of GW-correspondences between pairs $GWCor^{pair}_S$ over base $S$, as follows: 
objects of $GWCor^{pair}_S$
are pairs $(X,U)$ of smooth scheme $X$ over $S$ and open subscheme $U\subset X$,
and %fro two such pairs $(X,U)$, $(Y,V)$
the group of morphisms
\begin{multline*}GWCor^{pair}((X,U), (Y,V)) = \\H[
GWCor(X,V) \xrightarrow{d_0} GWCor(U,V)\oplus GWCor(X,Y)\xrightarrow{d_1} GWCor(U,Y)
],\end{multline*}
where 
$d_0 = (-\circ i,j\circ -)$, $d_1=(j\circ -,-\circ i)$,
$i\colon U\hookrightarrow X$, $j\colon V\hookrightarrow Y$,
and $H$ denotes cohomology in the middle term, i.e. $\Ker(d_1)/\Image(d_0)$.

Equivalent $GWCor^{pair}_S$ can be defined as the factor category in the full subcategory of the category of arrows in $GWCor_S$
spanned by open embeddings and factorised by the ideal consisting of morphisms $(\Phi,\widetilde\Phi)\colon (X,U)\to (Y,V)$ such that there is a lift $\Theta \colon X\to V$ in $GWCor_S$%, i.e. diagonal morphism in $GWCor_S$,
%such that diagram is commutative
$$\xymatrix{
X\ar[r]^{\Phi} & Y\\
U\ar@{^(->}[u]\ar[r]^{\widetilde{\Phi}}\ar@{-->}[ru] & V\ar@{^(->}[u]
.}$$ 
\end{definition}

\begin{remark}\label{rm:GWCorFpair}
For a homotopy invariant presheave with GW-transfers $\cal F$
the formula 
$$\begin{array}{ccl}
GWCor^{pair}&\longrightarrow& Ab\\
(Y,U)&\mapsto& \Coker(\cal F(Y)\to \cal F(U))
\end{array}$$
defines a homotopy invariant presheave on the category $GWCor^{pair}$.

\end{remark}

\begin{lemma}\label{lm:CorIm}
If $Q=(P,q)\in Q(\cal P(X,Y)$ is a quadratic space,
and $j\colon Y^\prime\hookrightarrow Y$ is an open immersion such that
$\Supp P=Y^\prime\times_Y \Supp P$,
then there is a unique quadratic space $Q^\prime\in Q(\cal P(X,Y^\prime))$,
such that $j\circ Q^\prime = Q$.
\end{lemma}\begin{proof}
%The claim 
The uniqueness follows from that the functor $(j_X)_*$, where 
$j_X\colon j\times id_X\colon Y^\prime\times X\hookrightarrow Y\times X$, is faithful.
The existence follows from that 
if $\Supp P=Y^\prime\times_Y \Supp P$, then 
$P$ defines the element in $\cal P(X,Y^\prime)$, and $q$ defines the quadratic form on it.

More precisely
let $j_X= j\times id_X\colon Y^\prime\times X\hookrightarrow Y\times X$,
and $(P^\prime,q^\prime)=j_X^*(P,q)$,
where by definition 
$P^\prime \in Coh(Y^\prime\times X)$ and $q\colon P^\prime\to D_X(P^\prime)$.
Then $\Supp P^\prime=\Supp P$ is finite over $X$ and
since $(j_X)_* ( j_X^*( P ) )= P$,
we have
${pr^\prime }_*( P^\prime )=pr_*( (j_X)_* ( j_X^*( P ) ) = pr_*(P)\in cal P(X)$,
where $pr^\prime\colon Y^\prime\times X\to X$ and $pr\colon Y\times X\to X$.

%where j_X=j\times id_X\colon Y^\prime\times X\hookrightarrow Y\times X$,
%and so $j_X^*(P)\in \cal P(X,Y^\prime)$ and $q$ defines
%=j_*(Q^\prime)

\end{proof}
\begin{lemma}\label{lm:QPairCor}
Let %$(X,X^\prime)$ and $(Y,Y^\prime)$ 
$j^X\colon X^\prime\hookrightarrow X$,  $j^Y\colon Y^\prime\hookrightarrow Y$
be open immersions, 
$Z=Y\setminus Y^\prime$, %be complement closed subscheme, and 
and $Q=(P,q)\in Q(\cal P(X,Y))$ be a quadratic space
such that 
%$$\Supp P\times_{X} X^\prime\subset Y\times_{Y^\prime} \Supp P,$$
$$Z\times_Y \Supp P\times_X X^\prime=\emptyset;$$
%(or equivalently $P\otimes_{\cal O(X)} \cal O(Y)
then $Q$ defines 
in a canonical way an element in $\Phi\in GWCor( (X,X^\prime), (Y,Y^\prime) )$,
such that $\Phi=[([Q],[Q^\prime])]$
and $(j^Y)_* (Q^\prime) = (j^X)^*(Q)$.
We will denote such element simply by $[Q]\in GWCor( (X,X^\prime), (Y,Y^\prime) )$.

%a unique element in $\Phi=GWCor( (X,X^\prime), (Y,Y^\prime) )$,
%such that $\Phi=[Q,Q^prime]$, %[([Q],[Q^\prime])]$
%and $(j^Y)_* (Q^\prime) = (j^X)^*(Q)$..
%We will denote such element simply by $[Q]\in GWCor( (X,X^\prime), (Y,Y^\prime) )$.

If moreover 
%$$\Supp P \subset Y_2\times_{Y_1} \Supp P,$$
$$Z \times_{Y} \Supp P=\emptyset,$$
then
$\Phi=0\in GWCor( (X,X^\prime), (Y,Y^\prime) )$. 
\end{lemma}\begin{proof}
Both statements immediate follows from lemma \ref{lm:CorIm}
applied for the first one to $(j^X)^*(Q)$ and for the second to $Q$.
\end{proof}

\subsection{Construction of quadratic space from a function on a curve.}

\begin{definition}\label{def:OrCurFinFun}
%We call by 
A \emph{curve} $C$ over a base scheme $S$ a scheme $C$ over $S$ such that dimension of all fibres of $C$ over $S$ is one, and we say that a curve $C$ is relative smooth if canonical morphism $C\to S$ is smooth.

%We call by a
An \emph{orientation} on smooth relative curve $C$ any trivialisation of its canonical class, i.e.
an isomorphism $\mu\colon \omega_S(C)\simeq \cal O(C)$.

We say that a regular function $f$ on relative curve $pr\colon C\to S$ is \emph{relatively finite} if
morphism $C \xrightarrow{(f, pr)}\affl\times S$ is finite.

So
%we call by
an \emph{oriented relative curve} $C$ 
with a relatively finite function $f$ with support $Z$
a a set
$(C,\mu,f,Z)$,
where $C\to S$ is a smooth relative curve,
$\mu$ is an orientation,
$f$ is a relatively finite function,
$Z=Z(f)$.
For a given finite scheme $Z$ over a scheme $S$, 
denote by $OrCur_S^Z$ % denotes 
the set of isomorphism classes of oriented curves with relatively finite function over $S$. %sets $(C,\mu,f,Z)$ of oriented curves with relatively finite function over $S$.
\end{definition}

\begin{proposition}\label{prop:constrOrCurFQCor}
There is map
$$\begin{array}{rccc}
\langle - \rangle \colon& OrCur^Z_S &\to & Q(\cal P(Z\to S))\\
& (\cal C,\mu, f , Z) &\mapsto & (k[Z],q) %(\cal O(Z), q)
,\end{array}$$
that is natural in respect to the base changes.
In other words this map %construction that for any 
takes any smooth oriented curve with relatively finite regular function $\cal C$ over $S$, 
with orientation $\mu\colon \omega_{\cal C}\otimes\omega^{-1}_S\simeq \cal O(\cal C)$ 
and relatively finite function $f\in k[\cal C]$ with vanish locus $Z$
%gives
to an invertible function $q$ in $k[Z]$ in a natural way over $S$.
%, i.e. in respect to base changes.
%quadratic %curve, $\cal C$ over $S$, 
\end{proposition}\begin{proof}
Consider the regular map $F= (f,pr_S)\colon \cal C\to \affl_S$.
By assumption $F$ is finite morphism of smooth schemes over $S$,
and since $\affl_S$ is affine then $\cal C$ is affine too
and $F$ is flat (by corollary V.3.9. and theorem II.4.7 \cite{AK}).
So we can apply proposition 2.1 from \cite{OP_WittPurity}
and get isomorphism 
$$\omega(\cal C)\simeq \cal Hom_{k[\affl_S]}(k[\cal C],k[\affl_S])$$
that respects base changes. %, i.e. that is natural in $S$.
Then using base change along the zero section $i_0\colon S\to \affl_S$
we get the isomorphism
$$q\colon k[Z]\simeq Hom_{k[S]}(k[Z],k[S])$$
that defines required quadratic space.

\end{proof}

\begin{lemma}\label{lm:sect-finfunct}
Suppose $\overline C\to S$ is a relative projective curve over a local base $S$,
%such that closed fibre of $C$ is irreducible. 
$\cal L$ is a very ample invertible seheaf on $\overline C$, and
$s,d\in \Gamma(\overline C, \cal L)$ are sections such that 
$Z(d)\cap E\neq \emptyset$, for each
irreducible component $E$ of the closed fibre of $\overline{C}$,
%$Z(d)$ contains at least one point in each irreducible component of closed fibre of $\overline{C}$
%is finite over $S$.
$Z(s)\neq \emptyset$, $Z(s)\cap Z(d)=\emptyset$;
then the function $s/d\colon \overline C-Z(d)\to \affl$ is relatively finite.
\end{lemma}
\begin{proof}
Consider the regular map 
$\overline{f}=([s\colon d],c)\colon \overline C\to \affl\times S$, where $c\colon \overline C\to S$ denotes the canonical morphism. The morphism $\overline{f}$ is projective and
$f=(s/d,c)\colon \overline C-Z(d)\to \affl\times S$ is the base change of $\overline{f}$ along the immersion $\affl\times S\hookrightarrow \prl\times \affl$,
hence $f$ is projective.

On the other hand since $d$ is section of very ample invertible sheaf $\cal L$,
it follows that $C-Z(d)$ is affine. Hence $f$ is affine morphism.
Thus $f$ is finite.
\end{proof}

\section{Quillen's trick and compactification}\label{sect:QuiltrCompact}

Here we summarise some some technical facts and geometric construction 
providing compactifications with ample bundles for some relative curves
used in the next sections. %to get compactifications with ample bundles . % constructed in the next section.
We start with following standard facts: %the following corollary of the Serre theorem (theorem 5.2, ch. 3 form \cite{Hs_AG}):
\begin{lemma}\label{lm:corSerreth}
Suppose $\cal O(1)$ is an ample invertible sheave on a scheme $X$, $Z\subset X$ is a closed subscheme,
then there is an integer $L$ such that for all $l>L$, 
the restriction homomorphism $\Gamma(X,\cal O(l))\to \Gamma(Z,\cal O(l)\big|_{Z})$ is surjective.
\end{lemma}
\begin{lemma}\label{lm:partcasecorSerreth}
Suppose 
$U$ is a local scheme, 
$\mathcal C\to U$ is a morphism of schemes, %with fibres being of dimension 1,
$\cal O(1)$ is an ample invertible sheave on $\mathcal C$, 
%Suppose $\cal O(1)$ is an ample bundle on a scheme $X$
and $Z_1,Z_2\subset X$ are closed subschemes such that
$Z_2$ is finite over $U$ and $Z_1\cap Z_2=\emptyset$;
then 
for some integer $L$, for all $l>L$, %integer $l$ that is bigger then some integer $L$, 
there is a section 
$s\in \Gamma(C,\cal O(l))$ such that $s\big|_{Z_1}=0$, and $s\big|_{Z_2}$ is invertible.
\end{lemma}
\begin{proof}
The claim of the first lemma is corollary of the Serre theorem \cite[ch. 3, theorem 5.2]{Hs_AG}. %(theorem 5.2, ch. 3 form \cite{Hs_AG}).
The second lemma follows form the first. 
\end{proof}

\begin{definition}\label{def:nisneigh}
Suppose $Z$ is a closed subscheme of a scheme $X$ over some base $S$.
A \emph{Nisnevich neighbourhood}
$\pi\colon (X^\prime,Z)\to (X,Z)$ 
is an etale morphism $\pi\colon X\prime\to X$ and closed subscheme $Z^\prime\subset X^\prime$ 
%of a closed subscheme $Z$ of a scheme $X$ over some base scheme $S$
%is a set $(\pi ,X^\prime,Z^\prime)$
%where $\pi\colon X\prime\to X$ is etale morphism and $Z^\prime$ is closed subscheme
such that $\pi$ induces an isomorphism $Z^\prime\simeq Z$.
%scheme $X^\prime$ with closed subscheme $Z^\prime$ and etale morphism $\pi\colon X\prime\to X$ such that $\pi$ induces an isomorphism $Z^\prime\simeq Z$.
For a pair of Nisnevich neighbourhoods 
$\pi_1\colon (X^\prime_1,Z_1)\to (X_1,Z_1)$ 
and 
$\pi_2\colon (X^\prime_2,Z_2)\to (X_2,Z_2)$, 
a \emph{morphism of Nisnevich neighbourhoods} $\pi_1\to \pi_2$
is a set of four morphisms 
$X^\prime_1\to X^\prime_2$ and $Z^\prime_1\to Z^\prime_2$
$X_1\to X_2$ and $Z_1\to Z_2$, 
such that corresponding cube is commutative 
(i.e. $\pi_1\to\pi_2$ is a morphism in the category of arrows of the category of arrows of the category of schemes).
%A \emph{morphism of Nisnevich neighbourhoods} $\pi_1\to \pi_2$, where 
%$\pi_1\colon (X^\prime_1,Z_1)\to (X_1,Z_1)$ 
%and 
%$\pi_2\colon (X^\prime_2,Z_2)\to (X_2,Z_2)$, 
%is a set of four morphisms 
%$X^\prime_1\to X^\prime_2$ and $Z^\prime_1\to Z^\prime_2$
%$X_1\to X_2$ and $Z_1\to Z_2$, 
%such that corresponding cube is commutative 
%(i.e. $\pi_1\to\pi_2$ is morphism in category of arrows in category of arrows in category of schemes).
\end{definition}\begin{definition}\label{def:compact}
%We call by 
A \emph{good relative compactification} of quasi-finite morphism of curves $\pi\colon X^\prime\to X$ over a base scheme $S$
the following set of data:
\begin{itemize}[leftmargin=15pt]
\item[1)]
a finite morphism $\overline{\pi}\colon \ovXp\to \ovX$ of projective schemes over $S$ 
with commutative diagram
in the category of schemes over $S$
$$\xymatrix{
\ovXp \ar[r]^{\ovpi}& \ovX \\
\pri X\ar[r]^{\pi}\ar@{^(->}[u]^{\pri j} & X\ar@{^(->}[u]^j
,}$$ 
where $j$ and $\pri j$ are open immersions
such that $\ovXp\setminus X^\prime$ and $\ovX\setminus X$
are finite over $S$;
\item[2)]
a very ample invertible sheave $\cal O(1)$ on $\ovX$,
such that $\ovpi^*(\cal O(1))$ is very ample too,
and 
a pair of sections $d^\prime\in \Gamma( \ovXp , \ovpi^*(\cal O(1)) )$, $d\in \Gamma(\ovX,\cal O(1))$, 
such that $\ovXp-Z(d^\prime)=\pri X$ and $\ovX-Z(d)=X$.
\end{itemize}
\end{definition}

\begin{lemma}\label{lm:relCur:comp:projection}
Suppose $\aff^d_k$ is an affine spaces of dimension $d$ over an infinite field $k$,
$z\in \aff^d_k$ is a point,
and $B,B_1\subset \aff^d_k$ are closed subschemes 
such that $B\neq \aff^d_k$, $codim\,B_1\geq 2$.
Then there is a linear projection $pr\colon \aff^d_k\to \aff^{d-1}_k$ such that 
the induced projection $B\to \aff^{d-1}_k$ is finite,
and $pr^{-1}(pr(z))\cap B_2=\emptyset$.%the fibre of $pr$ over $pr(z)$ doesn't intersect with $B_2$.
\end{lemma}
\begin{proof}
Both of the required conditions on the projection $pr$ defines 
Zariski open suschemes in the projective space of linear projections $\aff^d\to \aff^{d-1}$, which is equal to the infinity subspace $\mathbb P^{d-1}$ in $\mathbb P^d$.
Namely, $U_1$ is a complement to the intersection $\mathbb P^{d-1}\cap \overline B$
And let's denote the second one as $U_2$.

Then $U_1\neq\emptyset$, since $\dim(\mathbb P^{d-1}\setminus U_1)= \dim (\mathbb P^{d-1}\cap \overline B)\leq \dim\,B-1\leq d-2$.
Since the case field $k$ is infinite,
to prove that $U_2\neq\emptyset$ it is enough to consider the case of rational point $z\in \mathbb A^d_k$.
In such case $\mathbb P^{d-1}\setminus U_2$ is image of $B_1$ under the projection to the infinity subspace $\mathbb P^{d-1}$ with the center at the point $z$, and hence $\dim (\mathbb P^{d-1}\setminus U_2) \leq \dim B_2\leq d-2$. 

%To get the upper estimate to the dimension of the complement to the open subscheme $U_2\subset\mathbb P^{d-1}$ corresponding to the second condition
%it is enough to consider the case o
\end{proof}
Also we use following corollary of the Zariski main theorem (\cite{GD67} theorem 8.12.6).
\begin{proposition}\label{prop:ZarMianTh}
For any etale morphism $e\colon U\to Y$ there is a decomposition
$U\xrightarrow{u} X\xrightarrow{p} Y$, $p\circ u = e$ 
with $u$ dense open immersion and $p$ finite.
%such that $u$ is dense open immersion and $p$ is finite.
\end{proposition}
\begin{proof}
By the Zariski main theorem there is a decomposition 
$U\xrightarrow{\tilde u} \tilde X\xrightarrow{\tilde p} Y$
such that $\tilde u$ is open immersion and $\tilde p$ is finite.
Then if we put $X= Cl_{\tilde X}(U)$ (i.e. closure of $U$ in $\tilde X$) and define 
$u\colon U\to X$ and $p=\tilde p\big|_{X}$.
Then $u$ is dense open immersion and $p$ is finite, since it is composition of a closed embedding and a finite morphism. %; and $u$ is dense open immersion as required.
\end{proof}

\begin{lemma}\label{lm:et-ex:relCures}

Suppose $\pi \colon (X^\prime,Z^\prime)\to (X,Z)$
is Nishevich neighbourhood 
with smooth $X^\prime$ and $X$
and suppose $z^\prime\in Z^\prime$, $z\in Z$ are points such that $\pi(z^\prime)=z$;
then
there are a (relative) Nisnevich neighbourhood 
$\varpi\colon (\ovbX,\bZp)\to (\cal X,\cal Z)$ over an essential smooth local base $S$,
a morphism $\varpi\to \pi$
and a good relative compactification $\varpi\to \ovvarpi$
%, i.e.
%with commutative diagram
\begin{equation}\label{eq:reCurves:diag}\xymatrix{
X^\prime \ar[d]^{\pi} & \bXp\ar[d]^{\varpi}\ar[l]^{\pri v}\ar[r]^{\pri j} & \overline{\bXp}\ar[d]^{\overline\varpi} &\\
X & \cal X\ar[l]^v\ar[r]^j & \overline{\cal X} \ar[r] & S\\
%\bXp\ar[r]^{\varpi}\ar[u]^{\pri v} & \bX\ar[u]^v
}\end{equation}
%etale mprphism of relative smooth curves with trivial relative canonical classes
%$\varpi\colon \bXp\to \cal X$ over some essential smooth local base
%with commutative diagram
%$$\xymatrix{
%X^\prime \ar[r]^{\pi} & X \\
%\bXp\ar[r]^{\varpi}\ar[u]^{\pri v} & \bX\ar[u]^v
%}$$
%and points $\bzp\in bXp$ and $\cal z\in \cal X$, $\varpi(\bzp)=\cal z$,
such that :\begin{itemize}[leftmargin=15pt]
\item[1)] $v$ and $\pri v$ are pro-limits of open immersions,
\item[2)] there are lifts of $z^\prime$ and $z$ in $\cal X^\prime$ and $\cal X$,
\item[3)]
$\cal Z^\prime = v^{-1}(Z)$, $\cal Z = {\pri v}^{-1}(\pri Z)$ and both schemes are finite over $S$. 
\item[4)] $\bXp$ and $\cal X$ are smooth over $S$ and the relative canonical classes of both schemes are trivial.
\item[5)] %complements %closed subschemes 
$\ovbXp\setminus \bXp$ and $\ovbX\setminus \bX$ are finite over $S$.
\end{itemize}
%such that 
%$v^{-1}(Z)$ and ${\pri v}^{-1}(\pri Z)$ are finite over $S$ 
%and there are points $\bzp\in bXp$ and $\cal z\in \cal X$, $\varpi(\bzp)=\cal z$,
%$\pri v(\bzp) = \pri z$ and $v(\cal z)=z$

%(Precisely this means that,
%For etale morphism of smooth schemes $\pi\colon  X^\prime\to X$, 
%closed subscheme $Z\subset X$, such that $Z^\prime=\pi^{-1}(Z)\simeq Z$
%and points $z^\prime\in Z^\prime$, $z\in Z$, $\pi(z^\prime)=z$,
%%For Nisnevich neighbourhood $\pi \colon (X^\prime,Z^\prime)\to (X,Z)$ and point $z\in Z$
%there are 
%etale mprphism of relative curves 
%smooth with trivial relative canonical classes
%$\varpi\colon \bXp\to \cal X$ over some essential smooth local base $S$,
%and finite morphism $\overline{\varpi}\colon \overline{\bXp}\to \cal X$ of realtive curves over $S$ 
%and with commutative diagram as above satisfying conditions 1-5)

%Using definitions \ref{def:nisneigh} and \ref{def:compact} this means
%that 
%for any Nishevich neighbourhood $\pi \colon (X^\prime,Z^\prime)\to (X,Z)$,
%with smooth $X^\prime$ and $X$,
%there is (relative) Nisnevich neighbourhood $\varpi\colon (\ovbX,\bZp)\to (\cal X,\cal Z)$ over seesntial smooth local base with morphism $\varpi\to pi$
%and good relative compactification $\varpi\to \ovvarpi$.

\end{lemma}
\begin{proof}%[Proof of the lemma \ref{lm:et-ex:relCures}.]
Firstly let's replace $X$, $X^\prime$ and $Z$, $Z^\prime$ to some Zariski neighbourhoods of points $z$ and $z^\prime$,
in such way that canonical classes of new $X$ and $X^\prime$ are trivial.

Let $e\colon X\to \aff^d$ ($d=dim\,X$) be any etale morphism (which exists since $X$ is smooth).
Then using proposition \ref{prop:ZarMianTh} %main Zarisky theorem ( theorem .. , \site{EGA} ) 
we find a decomposition of the morphism $e$ into 
$X\xrightarrow{u} \overline{X}\xrightarrow{p} \aff^d$
such that $u$ is dense open immersion and $p$ is finite.
Again applying proposition \ref{prop:ZarMianTh} %Zarisky theorem ( theorem .. , \site{EGA} ) 
we find a decomposition of the morphism $u\circ \pi$ into
$\pri X\xrightarrow{\pri u} \overline{\pri X}\xrightarrow{\overline{\pi}} \overline{X}$, where
$\pri u$ is dense open immersion and $\overline{\pi}$ is finite.

Let's denote $B=\overline X \setminus X$, $B^\prime=\ovXp \setminus \pri X$,
$B_1 = \overline Z\cap B$, $B_1^\prime=\overline{Z^\prime}\cap B^\prime$,
where $\overline Z$ denotes the closure of $Z$ in $\overline X$ and similarly for $Z^\prime$.
Then 
since $X$ and $X^\prime$ are dense in $\overline X$ and $\overline{X^\prime}$, 
it follows that $dim\,B, dim\,B^\prime<d$
and $dim\,B_1,dim\,B_1^\prime\leq dim\,Z-1\leq d-2$.
Now %A
applying lemma \ref{lm:relCur:comp:projection}
to the closed subsets
$P=p(\pi(B^\prime)\cup B \cup \overline Z\cup \pi(\overline{Z}^\prime)),\;
P_1=p(B_1\cup B_2)$
%$$P=p(\pi(B^\prime)\cup B \cup \overline Z\cup \pi(\overline{Z}^\prime)),\;
%P_1=p(B_1\cup B_2)$$
and the point $p(z)\in \aff^d$,
we get 
a projection $pr\colon \aff^d\to \aff^{d-1}$, 
such that restriction to $P\to \aff^{d-1}$ is finite and $pr^{-1}(pr(e(z))\cap P_1=\emptyset$.

%So we get etale morphism $e\colon X\to \aff^d$ and projection $pr\colon \aff^d\to \aff^{d-1}$,
%such that $e$ and $e\circ \pi$ are finite over open subscheme $\aff^d\setminus P$, and $P$ is finite over $\aff^{d-1}$. 
Now consider 
base changes of $e$ and $\pi$ over local scheme $S=\aff^d_s$, $s = pr(p(z))\in \aff^{d-1}$.
We get etale morphisms 
$e_S\colon \cal X\to \affl_S$ and $\varpi\colon \bXp\to \cal X$ that
are finite over open subscheme $V = \affl_S \setminus P_S$, where $P_S=P\times_{\aff^{d-1}} S$ is finite over $S$.
And more over $\cal Z=Z\times_{\aff^{d-1}}S$ and $\cal Z^\prime=\pri Z\times_{\aff^{d-1}}S$ are finite over $S$ 
(since $\overline Z$, $\overline Z^\prime$ are finite over $\aff^d$ 
and since $\overline Z\setminus Z$ and $\overline Z^\prime\setminus \pri Z$ don't intersect with closed fibre).

%Now we c
Consider immersion $\affl_S\hookrightarrow \prl_S$
and twice applying proposition \ref{prop:ZarMianTh} %GD67Zariski main theorem %8.12.6  В EGAiv  Corollary 18.12.12. ("Main theorem" de Zariski)
we get decomposition 
$\cal X\xrightarrow{j} \overline{\cal X}\xrightarrow{\overline{e_S}} \prl_S$
and $\bXp\xrightarrow{\pri j} \overline{\bXp}\xrightarrow{\overline{\varpi}} \overline{\cal X}$.
$$\xymatrix{
& \affl_S\ar@{^(->}[d] &\cal X \ar[l]^{e_S}\ar@{^(->}[d]^{j}& \bXp\ar@{^(->}[d]^{j^\prime}\ar[l]^{\varpi}\\
S & \prl_S\ar[l] &\overline{\cal X}\ar[l]^{\overline{e_S}} & \overline{\bXp} \ar[l]^{\overline\varpi}
}$$

Then $\cal X\times_{\affl_S} V$ is dense in $\overline{\cal X}\times_{\affl_S} V$,
and in the same time it is finite over $V$; hence $\cal X\times_{\affl_S} V=\overline{\cal X}\times_{\affl_S} V$.
Whence 
$\overline{\cal X} \setminus \cal X\subset \overline{\cal X}\times_{\affl_S} (\prl_S\setminus V) = \overline{\cal X}\times_{\affl_S} (\infty_S\cup P_S),$
and consequently $\overline{\cal X} \setminus \cal X$ is finite over $S$.
Similarly $\overline{\bXp}\setminus \bXp$ is finite over $S$.

Thus we construct diagram \eqref{eq:reCurves:diag} satisfying the required properties, and
it is enough to find very a ample sheave $\cal O(1)$ and sections required in definition \ref{def:compact}. 

Since $\ovtX$ is projective over the local scheme $S$,
there is 
an ample sheave $\cal O(1)$ on $\ovtX$.
Then since $\ovvarpi$ is finite,
$\ovvarpi^*(\cal O(1))$
is ample sheave on $\overline{X^\prime}$.
Then for some $l$ 
sheaves
$\cal O(l)$ and $\ovvarpi^*(\cal O(l))$
are very ample, %bundle on $X$ too,
and let's replace $\cal O(1)$ by $\cal O(l)$.

%$dim\,Z(d)_{red} \leq dim\,\overline C-1=dim\,S$
%$Z(d)$ 

Now 
by lemma \ref{lm:corSerreth} for some integer $l_0$ there is
a section $r\in \Gamma(\ovbX, \cal O(l_0))$, $Z(r)\cap (\ovbX-\cal X) = \emptyset$.
Again by lemma \ref{lm:corSerreth} 
there is an integer $L$ such that for all $l>L$ 
there is a section 
$d\in \Gamma(\ovbX, \cal O(l_1))$ 
such that 
$d\big|_{\ovbX-\cal X}=0$, $d(z)\neq 0$, $d\big|_{Z(r)}$ is invertible.
Then on the one side $Z(d)\subset \ovbX-Z(r)$, and hence it is affine. And on the other side $Z(d)$ is a closed subscheme in the projective scheme $\ovtX$ over $S$,
and hence $Z(d)$ is projective over $S$.
Thus $Z(d)$ is finite over $S$.

Next similarly 
there is an integer $L^\prime$
such that for all $l>L^\prime$
there is a section
$d^\prime \in \Gamma(\ovbX, \cal O(l_1))$ such that 
$d^\prime\big|_{\ovbXp-\bXp}=0$, $d^\prime(z)\neq 0$, $Z(d^\prime)$ is finite over $S$.
Finally replacing $\mathcal O(1)$ to $\mathcal O(l)$ for any $l>L,L^\prime$ and replacing $\cal {X}$ to $\cal {X}-Z(d)$, and 
$\bXp$ by $\bXp-Z(d^\prime)$, we get the claim.
%we fund a section 
%and such that $Z(d^\prime)$ is finite over $S$, and replace $\bXp$ by $\bXp-Z(d^\prime)$.
%
%So we can replace $\cal {X}$ by $\cal {X}-Z(d)$.

%Since 
%$e_S\times_{\affl_S} V$ and $\pi_S\times_{\affl_S} V$
%are finite  
%$\cal X\times_{\affl_S} V$ and $\bXp\times_{\affl_S} V$ are  finite over $V$

\end{proof}

\section{Constructions of functions on relative curves}\label{sect:FunctionsConstr}

In this section we present the constructions of relative curves with trivialisation of the relative canonical class ans a regular functions on the curves, which are used in the proofs of excision and injectivity theorems.

\begin{lemma}\label{lm:sect_inj,nonmut}
Suppose $\pi\colon \pri X\to X$ is a finite morphism of projective curves over an infinite field $k$,
and $\mathcal O(1)$ is a vary ample line bundle on $\pri X$;
suppose 
$z_i\in \pri X$ is a closed point, 
$Y\subset \pri X$ is a closed subscheme,
$z\not\in Y$, $\pri X-Y$ is smooth;
then 
%%for any set of closed points $z_i\in \pri X$,
%for any closed point $z_i\in \pri X$
%%such that images $\pi(z_i)$ are distinct, 
%and closed subscheme $Y\subset \pri X$, 
%%such that $z_i\not\in Y$, and $\pri X-Y$ is smooth, 
%such that $z\not\in Y$, and $\pri X-Y$ is smooth, 
for any invertible section $s_Y\in \Gamma(Y,\mathcal O(1))$
%let $z\in \pri X$ be closed point, $Y\subset \pri X$ be closed subscheme, $z\not\in Y$,
%for some $n\in \mathbb N$ and 
there is an integer $L$ such that for all $l>L$
%for all $l$ bigger some $L$,
there is a global section $s\in \Gamma(\pri X,\mathcal O(l))$
such that \begin{itemize}[leftmargin = 15pt]
\item[1)]
$s(z)=0$ and $s\big|_{Y}=s_Y^{l}$, %is invertible,
\item[2)]
$\divisor s$ hasn't multiple points (or equivalently $Z(s)$ is reduced),
\item[3)] 
$\pi$ induces the closed injection $Z(s)\to X$.
\end{itemize}

%Пусть $\pi\colon X^\prime\to X$ --- конечный морфизм проективных кривых над бесконечным полем, $z$ ---  замкнутая точка $X^\prime$,  $Y$ --- замкнутая подсхема $X^\prime$, не содержащая $z,$ и $\mathcal L$ --- очень обильный локально свободный пучок ранга 1 на $X^\prime$. Тогда для всех $n$, больших некоторого $n_0$, существует глобальное сечение $s$ пучка $\mathcal L^{\otimes n}$, обращающееся в $0$ в $z,$ не обращающееся в $0$ на $Y$ и такое, что ограничение $\pi$ на $Z( s)$ является замкнутым вложением (точнее, имеется в виду ограничение $\pi$ на подсхему в $X^\prime$, определяемую пучком идеалов в $\bO(X^\prime)$, состоящим из функций $f:\; div f\geqslant div s$).
\end{lemma} 
\begin{proof}
In short the claim follows from 
that for a big $l$
the first condition on $s$ defines some non-empty affine subspace $\Gamma$
in the affine space of global sections of $\mathcal O(l)$ and $s(z)$,
and the second and third conditions defines non-empty open subsets in $\Gamma$.
and for big $l$ this conditions define subscheme of codimension at least 1.
%In short the claim follows from 
%that for a big $l$
%the first condition on $s$ defines some non-empty affine subspace $A$
%in the affine space of global sections of $\mathcal O(l)$ and $s(z)$,
%and the second and third conditions defines non-empty open subsets in $A$.
%%both required conditions conditions that $\divisor  s$ has multiple points and that restriction $\pi\big|_{Z(s)}$ isn't injection define closed subscheme in (non-empty) affine space of sections that are zero at $z$ and are equal to some invertible function on $Y$,
%and for big $l$ this conditions define subscheme of codimension at least 1.
%%codimension subscheme of sections $\{s\colon \divisor s\geq 2\pri x\}$ 
%%for each point $\pri x\in X^\prime$,
%%and
%%codimension subscheme of sections $\{s\colon \divisor s\geq 2\pri x\}$ and
%%for each point $\pri x\in X^\prime$,
%%are at least 2,
%%and  has at each $dim\,X^\prime=dim\,X=1$ then both .

Indeed, 
%using lemma \ref{lm:corSerreth}
let's choose some section $d\in \mathcal O(1)$ %for some $n$ 
that is invertible on $z\cup Y$, %(using lemma \ref{} for example)
fix some lift of $s_Y$ to an invertible section $\overline{s}$ of $\mathcal L(n)$ on $Y\amalg Z(d)$
and consider for each $l$ the affine subspace 
$$\Gamma=\{s\in \Gamma(\mathcal X^\prime,\mathcal O(l))\colon s\big|_{Y\amalg Z(d)}=\overline s^l, s(z)=0\}.$$ 
%$\Gamma$ of global sections 
%$s\in \Gamma(\mathcal X^\prime,\mathcal O(l)\colon s\big|_{Y\cup Z(d)}=\overline s, s(z)=0$. 

Let $s_\Gamma\in \Gamma(X^\prime\times\Gamma,\mathcal O(l))$ be the universal section.
Consider the closed subschemes %intersection 
\begin{gather*}
B_1= Z(s_\Gamma) \cap Z(d\,(s_\Gamma/d^l))\subset \pri X\times\Gamma,
%B_2=\Supp 
%\Coker[ \mathcal O(X) \xrightarrow{\varepsilon_{\pi}} p_*(\mathcal O(Z(s_\Gamma)) ]
%\subset X\times\Gamma,
\end{gather*}
where $d\,(s_\Gamma/d^l)$ denotes 
differential of regular function $s_\Gamma/d^l$ on $\pri X-Y$ that is section of $\omega(\pri X)$
(image is closed since $B_1\subset Z(s_\Gamma)$ that is finite over $\Gamma$),
and 
$$B_2=\Supp 
\Coker[ \pi^*(\mathcal O(X))\oplus \mathcal O(\pri X) \xrightarrow{(\varepsilon_{\pi}, s_\Gamma)} \pi\mathcal O(l) ],$$ 
where 
$p\colon Z(s_\Gamma)\to X\times\Gamma$ be composition of injection into $\pri X\times\Gamma$ and $\pi\times id_\Gamma$
and
$\varepsilon_{p}$ denotes unit of the adjuniction $(p^*,p_*)$
(image is closed since $B_2\subset (\pi\times id_{\Gamma})_*(Z(s_\Gamma))$ is finite over $\Gamma$).

Then the subset of such sections $s$ that the zero divisor $Z(s)$ has multiple points 
is contained in the image of $B_1$
under the projection 
$\pri{pr}\colon \pri X\times\Gamma\to\Gamma$. %$\pri{pr}_\Gamma\colon \pri X\times\Gamma\to\Gamma$ 
And subset of such sections $s$ that $\pi\big|_{Z(s)}$ isn't injection,
is contained in image of $B_2$ under the projection %along projection 
$pr\colon  X\times\Gamma\to\Gamma$.
%of closed subscheme 
%So to prove that there is required it is enough to show that 
%$$\pri{pr}_\Gamma(B_1)\cup pr_\Gamma(B_2)\neq \Gamma.$$
So to prove the claim it is enough to find at least one rational point in $\Gamma- (\pri{pr}(B_1)\cup pr(B_2))$.

Since base filed $k$ is infinite, to get the claim it is enough to show that 
dimensions of $\pri{pr}(B_1)$ and $pr(B_2)$ 
are less then $\dim \Gamma$,
and since base changes along filed extensions doesn't change dimension,
it is enough to consider the case of algebraically closed filed $k$. % base change from $k$ to the algebraic closure $k/k$. 

Let's note that for $l$ larger some $L$
for each pair of points $x_1,x_2\in X^\prime$ 
the restriction homomorphism
$$
r_{x_1,x_2,l}\colon 
\Gamma(\pri X,  \mathcal O(l) )\to \Gamma(S(x_1+x_2+z+Y+\divisor d),\mathcal O(l))
$$
is surjective.
Indeed, for each $l$ 
surjectivity of  $r_{x_1,x_2,l}$ 
is open condition on the pair $(x_1,x_2)$, 
and lemma \ref{lm:corSerreth} implies that for each pair of points $(x_1,x_2)$ there is some $L_{(x_1,x_2)}$ such that for all $l>L_{(x_1,x_2)}$, $r_{(x_1,x_2,l)}$ is surjective. 
%Choose some such $L$

Then 
for $l>L$,
$$
\Gamma\twoheadrightarrow 
\{s\in\Gamma(S(x_1+x_2+z+Y+\divisor d),\mathcal O(l))\colon\,s\big|_{Y\amalg Z(d)}=\overline s, s(z)=0 \}\simeq{k}^2
$$
and  for any pair $x_1,x_2\in X^\prime-(Y\cup Z(d))$
\begin{equation*}
%\codim_\Gamma\{s\in \Gamma|\, \divisor s\geq 2x\}\geq 2, \forall x\in \pri X;\;
\codim_\Gamma\{s\in \Gamma|\,\divisor  s\geqslant x_1+x_2\}=
\begin{cases}2 \text{, if } x_1,x_2\neq z, \\1 \text{, otherwise.}\end{cases}
\end{equation*}
%\{f\in k[S(x_1+x_2+z)]:\;div f\geqslant x_1+x_2, div f\geqslant z\}$ в пространстве функций,   обращающихся в $0$ в $z,$ 
%if $x_1,x_2\neq z$, it is equal to 1, otherwise.
Hence 
for almost all points $x\in \pri X$, the dimension of the fibre of $B_1$ over $x$ is at least 2, 
and 
for all $x\in X$, the dimension of the fibre of $B_2$ over $x$ is at lest 2.
%Then consequently 
Thus %consequently
$$
\codim_{\Gamma}\pri{pr}_{k}(B_1 )\geq 2-dim\,X=1,\quad
\codim_{\Gamma} pr_{k}(B_2)\geq 2-dim\,X^\prime>1.
$$

%Thus $\pri{pr}( B_1)\cup pr(B_2 )\neq \Gamma$. % over $k$. %${k}$ and hence over $k$ too,
%and since field $k$ is infinite then that is rational point in in $\Gamma- (\pri{pr}(B_1)\cup pr(B_2))$.
%%$\Gamma$ outside of $pr_{\Gamma}( B_1\cup B_2 )$, that is required section.

\end{proof}

\begin{lemma}\label{lm:et-ex:Zsect}
%Let $\overline{\cal C}$ be variety and $\cal C$ be smooth open subscheme,
%let $\overline{\cal C}\to S$ be  projective morphism of relative dimension one,
%let $\cal L$ be  ample bundle
%with global suction $d$, such that $\cal C = \overline{\cal C}-Z(d)$,
%let $B$ be finite subscheme, $\delta$
%$Z_1$, $Z_2$
%there is 
%section 
%For a variety $\overline{\cal C}$ with smooth open subscheme $\cal C$
%with projective morphism $\overline{\cal C}\to S$ of relative dimension one
%and ample bundle $\cal L$
%with global suction $d$, such that $\cal C = \overline{\cal C}-Z(d)$

Suppose
$\varpi\colon (\pri{\cal X},\pri{\cal Z} )\to (\cal X,\cal Z)$
is a Nisnevich neigbourhood
over a local base scheme $S$ over an infinite field $k$ (see def. \ref{def:nisneigh})
such that $\cal Z$ is finite over $S$,
and $\overline{\varpi}\colon \ovbXp\to \ovbX$ is a good relative compactification (see def. \ref{def:compact});
suppose $z\in \pri X$ is closed point and
$\Delta$ is primitive divisor in $X$ finite over $S$,
such that closed fibre of $\Delta$ is $z$;
then there is an integer $L$ such that for any $l>L$ there is a
section $s\in \Gamma(\cal O(l))$ such that\begin{itemize}[leftmargin=15pt]
\item[1)] $Z(s\big|_{ \ovvarpi^{-1}( {\cal Z^\prime} ) } )\simeq \pi\big|_{\cal Z^\prime}^{-1} ( i_{\cal Z}^*(\Delta) )$,
where $i_{\cal Z}\colon \cal Z\to X$ denotes the closed injection 
(so the term $\pi\big|_{\cal Z^\prime}^{-1} ( i_{\cal Z}^*(\Delta) )$ is equal to $\Delta\cap \cal Z$ considered as subscheme in $\pri{\cal X}$ using the isomorphism $Z^\prime\simeq Z$.)
%(and equivalently  $i_{\cal Z}^*(\Delta) = \Delta\times_X {\cal Z}$). 
\item[2)] $Z(s)\subset \bXp$, $Z(s)$ is reduced and irreducible components $Z(s)$ don't intersect to each other,
\item[3)] $Z(s)\to X$ is closed injection.
\end{itemize}
\end{lemma}
\begin{proof}

%${\pri Z}\times_S {\cal X_x} \simeq Z\times_S {\cal X_x}$
%$\cal L(\Delta)$ on $\cal X$
%$\cal O(l)$ on $\cal X^\prime$
%$$\cal O(l)\big|_{ {\pri Z}\times_S {\cal X_x} } \simeq 
%\cal L(\Delta)\big|_{ Z\times_S {\cal X_x} }$$
%
%$s(p)\neq 0$ for $p\in \pi^{-1}(\cal Z) - \pri{\cal Z}$

Since $\cal Z$ is finite over the local scheme $S$, it is semi-local. Hence any line bundle on $\cal Z$ is trivial.
Since $\varpi$ induce the isomorphism $\cal Z^\prime\simeq \cal Z$,
for each integer $l$ there is some isomorphism
\begin{equation}\label{eq:et-ex:lmZsect:identDelta}
\cal O(l)\big|_{ \pri{\cal  Z} } \simeq 
\varpi^*( \cal L(\Delta)\big|_{ \cal Z } )
.\end{equation}
Let $\delta\in \Gamma(\cal Z,\cal L(\Delta))$ be a section such that $div\,\delta=\Delta$, and since $\Delta$ is primitive then $Z(\delta)=\Delta$.
Then using isomorphism \eqref{eq:et-ex:lmZsect:identDelta} for each $l$ we get some section
$$\delta^\prime\in \Gamma(\cal Z^\prime,\cal O(l))\colon 
Z(\delta^\prime)\simeq \pi\big|_{\cal Z^\prime}^{-1} ( i_{\cal Z}^*(\Delta) ).$$

%$s(p)\neq 0$ for $p\in \pi^{-1}(\cal Z) - \pri{\cal Z}$
%s\big|_{ \varpi^{-1}( {\cal Z} ) } 
Let $\pi\colon (X^\prime,Z^\prime)\to (X,Z)$ and $\ovpi\colon \ovXp\to \ovX$
denotes the fibre of $\varpi$ over the closed point of $S$.
Let $z^\prime$ denotes preimage of $z$ in $Z^\prime$.
%$X^\prime$, $Z^\prime$, $X$, $Z$ denotes fibres of $\cal X^\prime$, $\cal Z^\prime$, $\cal X$ and $\cal Z$ over closed point of $S$.
Using lemma \ref{lm:sect_inj,nonmut} for all $l$ bigger some $L$ we find a section $s_{cf}\in \Gamma(\ovXp,\cal O(l))$ such that 
\begin{itemize}
\item[1)]
$s_{cf}\big|_{   Z^\prime  } = \delta^\prime\big|_{Z^\prime}$, $s_{cf}\big|_{\ovpi^{-1}(Z)-Z^\prime}$ is invertible,
\item[2)]
$Z(s_{cf} ) \subset X^\prime$, $div\,s_{sf}$ hasn't multiple points (or equivalently $Z(s_{cf})$ is reduced),
\item[3)] 
$\pi$ induce closed injection $Z(s_{cf})\to X$.
\end{itemize}
(To apply lemma \ref{lm:sect_inj,nonmut} we set
$z=z^\prime$, $Y=(Z^\prime - z^\prime) \cup (\ovXp\setminus \pri X)$, and 
$s_Y$ is any lift of $\delta^\prime$ to invertible section of $Y$.)

Let $s$ be any lift of $s_{cf}$ and $\delta^\prime$ to a global section on $\ovbX$,
that exists by lemma \ref{lm:corSerreth}, since $s_{cf}$ and $\delta^\prime$ are agreed on $Z^\prime$.

Since $s_{cf}\big|_{\ovpi^{-1}(Z)-Z^\prime}$ is invertible, it follows that $s\big|_{\ovvarpi^{-1}(\cal Z)-\bZp}$ is invertible too.
Hence $Z( s\big|_{\varpi^{-1}(\cal Z)} )= Z( s\big|_{\bZp} ) = Z(\delta^\prime)=
\pi\big|_{\cal Z^\prime}^{-1} ( i_{\cal Z}^*(\Delta) )$,
and so the first claim of the lemma on $s$ holds.

By Nakayama's lemma
if the closed fibre of the intersection of two subschemes of a scheme over a local base is empty, then the intersection if these schemes is empty at all.
So the second claim on $s$ follows form that $div\,s_{cf}$ hasn't multiple points.

And since by Nakayama's lemma again, injectivity of $\pi\big|_{Z(s_{cf})}\colon Z(s_{cf})\to X$ implies that 
$\varpi$ induce injection $Z(s)\to \cal X$, 
it follows that
%and so 
the claim of the third point of lemma holds. % for $s$ holds.

\end{proof}

\begin{lemma}\label{lm:et-ex:inj:FinFuncSup}
Let 
$\pi\colon \bXp\to \cal X$ be a morphism 
of relative projective curves over a local base $U$,
with a good compactification $(\ovpi\colon \ovbXp\to \ovbX,\cal O(1))$
(see def. \ref{def:compact}),
let
$\cal Z\subset\bX$, $\cal Z^\prime\subset\bXp$ be closed subschemes, and
$\Delta\subset \cal X$ be a primitive divisor 
such that closed fibre of $\Delta$ is a point $z\in \cal X$
and \begin{equation}\label{eq:lm:et-ex:i:FF:Delta} Z\times_{\bX}\Delta\times_U (U-Z)=\emptyset.\end{equation}
Then there are 
regular functions $f_0,f_1\in k[\cal X]$ and $f\in k[\cal X\times\affl]$
and a regular map $l\colon Z(f_0)\to \bXp$ such that \begin{itemize}[leftmargin=15pt]
\item[1)]
$f_0,f_1$ and $f$ are relatively finite over $U$ and $U\times\affl$ (see def. \ref{def:OrCurFinFun}),
\item[2)]%and such that 
$i_0^*(f)=f_0, i_1^*(f)=f_1$
where $i_0,i_1\colon U\to U\times\affl$ denotes zero and unit sections,
\item[3)]
$Z \times_X Z(f)\times_{U} (U-Z_z)=\emptyset$,
%\label{eq:et-ex:Zfpair}
\item[4)]
$Z(f_1)=\Delta\amalg (Z(f_1)-\Delta),\, Z\times_{X} (Z(f_1)-\Delta)=\emptyset$
%\label{eq:et-ex:Z1}
\item[5)]
$l$ is a lift of the canonical injection $i_{Z(f_0)}\colon Z(f_0)\hookrightarrow \bX$, i.e.
$\varpi \circ l=i_{Z(f_0)},$
\item[6)]
$Z^\prime \times_{X^\prime} l(Z(f_0))\times_U (U-Z_z)=\emptyset$.
\end{itemize}
\end{lemma}
\begin{proof}
Denote $D = \ovbX\setminus \cal X$, $D^\prime = \ovbXp\setminus \bXp$,
and
denote 
%$\Delta$ graph of canonical embedding 
%$X_z=\tilde X_z\hookrightarrow \tilde X$ as closed subset in $\ovbX$.
%and 
$\Delta^\prime = \varpi\big|_{\bZp}^{-1}( i_{\cal Z}^{-1}(\Delta) )$, 
where $i_{\cal Z}\colon \cal Z\to \cal X$ denotes the canonical injection
(so $\Delta^\prime \subset \bXp$ is the image of $\Delta\cap \cal Z$ after the identification $\bZp\simeq \cal Z$).

%And denote $D = \ovbX\setminus \cal X$, $D^\prime = \ovbXp\setminus \bXp$.

%Next 
Applying lemma \ref{lm:et-ex:Zsect}
we find $L_0$ such that for any $l>L_0$ there is 
a global section 
$\pri s\in\Gamma(\ovbXp,\cal O(l))$ 
such that
the restriction $\varpi\big|_{Z(\pri s)}\colon Z(\pri s)\to \cal{X}^\prime$ is injective, $Z(\pri s\big|_{\cal Z})=\Delta^\prime$,
%$$\pri s\in\Gamma(\ovbXp,\cal O(l)), 
%\varpi\big|_{Z(\pri s)}\colon Z(\pri s)\hookrightarrow \cal{X}^\prime, Z(\pri s\big|_{\cal Z})=\Delta^\prime$$
%And applying lemma \ref{}
%we find $l_1$ such that for any $l>l_2$ there is global section 
and such that $Z(\pri s)$ is reduced and it is the disjoint union of irreducible components. %its irreducible components don't intersect.

Hence $Z(\pri s)$ is disjoint union of irreducible closed subschemes in $\ovbXp$ of codimension one.
The direct image effective divisor $div\,\pri s$ along $\varpi$
is in the linear system of the line bundle $\cal O(l n)$, where $n= deg\,\ovvarpi$.
Then there is a global section
$$s_0 \in\Gamma(\ovbX,\cal O(l n)), div\,s_0=\varpi_*(div\,s)$$,
and then $Z(s_0)$ is equal to the image of $Z(\pri s)$ under the injection defined by $\varpi$.
Moreover since $Z(\pri s\big|_{ \ovvarpi^{-1}(\cal Z) })$ is equal to $i_{\cal Z}^{-1}(\Delta)$,
it follows that $Z(s_0\big|){\cal Z})=i_{\cal Z}^{-1}(\Delta)$.
%\varpi\big|_{Z(s)}\colon Z(s)\hookrightarrow \cal{X}^\prime, Z(s\big|_{\cal Z})=\Delta^\prime.$$
%$$s_0\in(\ovbXp,\cal O(l))\colon s_1\big|_{D\cup Z\times_S {X_x}}=s_0$$
%$s\big|_{\Delta}=0$, $s\big|_{

Thus for each $l>L_0$ we have find 
%closed subscheme and section and lift morphism
%$$
%Z_0\subset \cal X, i_{\cal Z}^{-1}(Z_0)=i_{\cal Z}^{-1}(\Delta),
%s_0 \in\Gamma(\ovbX,\cal O(l n))\colon Z(s)=Z_0.$$
\begin{equation}\label{eq:lm:et-ex:i:FF:lift}
Z_0\subset \cal X, l\colon Z_0\to \bXp, 
s_0 \in\Gamma(\ovbX,\cal O(l n))\colon 
Z(s)=Z_0,
%i_{\cal Z}^{-1}(Z_0)=i_{\cal Z}^{-1}(\Delta),
\cal Z\cap Z_0=\cal Z\cap\Delta,
\pi\circ l=i_{\cal Z}.\end{equation}
Next 
%using 
by lemma \ref{lm:corSerreth} 
there is $L_1$ such that for all $l>L_1$ %larger some $L_1$
we can find a section 
$$s_1\in\Gamma(\ovbX,\cal O(l n))\colon s_1\big|_{D\cup \cal Z}=s_0\big|_{D\cup \cal Z}, s_1\big|_{\Delta}=0.$$
Define 
\begin{gather*}
s=t s_1+(1-t) s_0\in\Gamma(\ovbX,\cal O(l n)),\\
f_0=s_0/d^{l n}\in k[\bX] , f_1=s_1/d^{l n}\in k[\bX] , f= s/d^{l n}\in k[\bX\times\affl].
\end{gather*}
%$$\tilde{s}\in\Gamma(\ovbX,\cal O(l n))\colon \tilde{s}=t s_1+(1-t) s_0.$$

By lemma \ref{lm:sect-finfunct}
functions 
$f_0$ , $f_1$ , $f$ 
%$f_0=s_0/d^{l n}$ , $f_1=s_1/d^{l n}$ , $f=\tilde s/d^{l n}$ 
are relatively finite over $S$ (see def. \ref{def:OrCurFinFun}), so the point 1 holds. %condition 1 holds.
%The condition 2 follows from the definition. 
The point 2 follows from the definition. 

%The conditions 5) and 6) 
The point 5 and 6 follows from \eqref{eq:lm:et-ex:i:FF:lift} and from \eqref{eq:lm:et-ex:i:FF:Delta}.% assumption on $\Delta$ \eqref{eq:lm:et-ex:i:FF:Delta}.
%The condition 3) 
The point 3
follows from that $s\big|_{\cal Z}=s_0\big|_{\cal Z}$ and from \eqref{eq:lm:et-ex:i:FF:lift} and \eqref{eq:lm:et-ex:i:FF:Delta} too.

To get the claim it is enough to check the point 4. %condition 4).
Let $\delta\in \Gamma(\ovbX,\cal L(\Delta))\colon div\,\delta=\Delta$.
Since  by definition $s_1\big|_{\Delta}=0$, it follows that %then 
$s_1/\delta\in \Gamma(\ovbX,\cal L(-\Delta)(l))$ is regular section.
Then since 
$Z(s_1\big|_{\cal Z})=\cal Z\cap \Delta$,
$s_1/\delta\big|_{\cal Z}$ is invertible,
and since the closed fibre of $\Delta$ is point $z$ contained in $\cal Z$, 
%then 
%this implies that 
the above implies that
$s_1/\delta\big|_{\Delta}$ is invertible too, and hence we get $Z(s_1/\delta)\cap (\cal Z\cup \Delta)=\emptyset$.
Thus 
$$
Z(s_1) = \Delta\amalg Z(s_1/\delta), \cal Z\cap Z(s_1/\delta)=\emptyset
.$$
%This follows from that 
%$$Z(s_1\big|_{\cal Z})=\cal Z\cap \Delta, s_1\big|_{\Delta}=0.$$
%Indeed,

\end{proof}

\begin{lemma}\label{lm:et-ex:sur:FinFuncSup}
Let $U$ be a local scheme and
$\pi\colon (\pri U,\pri Z)\to (U,Z)$ be a Nisnevich neighbourhood 
of a closed subscheme $Z\subset U$.
%morphism 
%local schemes $U$,
Let $\ovbX$ be a projective curve over $U$ with a very ample bundle $\cal O(1)$,
let $d\in \Gamma(\ovbX,\cal O(1))$, $\bX=\ovbX-Z(d)$,
%with good compactification $(\ovpi\colon \ovbXp\to \ovbX,\cal O(1))$
%(see def. \ref{def:compact}),
%and 
%Let
and let 
$\cal Z\subset\bX$ be a closed subscheme finite over $U$. 
Let 
$\ovbXp=\ovbX\times_U \pri U$, $\bXp=X\times_U\pri U$, $\bZp=\cal Z\times_U \pri U$:
$$\xymatrix{
& \bZp\ar@{^(->}[d]\ar[r]& \cal Z\ar@{^(->}[d]\\
\pri\Delta\ar@{^(->}[r] & \ovbXp \ar[d]\ar[r]^{\phi}&\ovbX\ar[d]\\
& \pri U\ar[r]^\pi & U.
}$$
%$\verepsilon = (id_{\pri X}\times_U \pi) \colon \pri X\times_U\pri U\to \pri X$
Let 
$\Delta^\prime\subset \bXp$ 
be a primitive divisor finite over $\pri U$,
%that closed fibre is closed point $z\in \cal X$.
and $\delta\in k[\cal Z]$ be a regular function such that
\begin{equation}\label{eq:lm:et-ex:sur:FF:condDelta}
(\bZp\cap \pri\Delta)\times_{\pri U} (\pri U-\pri Z)=\emptyset,\;
(\cal Z\cap Z(\delta))\times_{U} (U-Z)=\emptyset,\;
Z(\delta)\times_U \pri U=\pri Z\cap \pri\Delta.
.\end{equation}

Then there are 
%regular functions $f_0,f_1\in k[\cal X]$ and $f\in k[\cal X\times\affl]$
%and regular map $l\colon Z(f_0)\to \bXp$, such that 
a relatively finite (over $U$) regular function $\pri f\in k[\bX]$
%a regular function 
%$\pri f\in k[\bX]$,
%that is relatively finite over $U$,
a relatively finite (over $\pri U$ and $\pri U\times\affl$) and functions 
$f_0,f_1\in k[\bXp]$,
$f\in k[\bXp\times\affl]$
such that\begin{itemize}[leftmargin=15pt]
\item[1)] 
$\pi_z^*(f^\prime)=f_0,\, i_0^*(f)=f_0,\, i_1^*(f)=f_1, $
\item[2)]
$
\pri Z \times_{\pri X} Z(f)\times_{\pri U} (\pri U-\pri Z)=\emptyset
%\label{eq:et-ex:sur:Zfpair}
$
%$
%\pri Z \times_{\pri X} Z(f)\times_{\pri U\times\affl} (\pri U-\pri Z)\times \affl=\emptyset
%%\label{eq:et-ex:sur:Zfpair}
%$
\item[3)]
$
Z^\prime \times_{X^\prime} Z(f^\prime) \times_{U} (U-Z)=\emptyset
$
%\label{eq:et-ex:sur:Zppair}
\item[4)]
$Z(f_1)=\pri\Delta\amalg (Z(f_1)-\Delta),\, (\pri Z)\times_{\pri X} (Z(f_1)-\pri\Delta)=\emptyset
%\label{eq:et-ex:sur:Z1}
$
\end{itemize}

\end{lemma}\begin{proof}
Using lemma \ref{lm:corSerreth} we choose $l$ such that 
\begin{equation}\label{eq:lm:et-ex:sur:FinFunc:SurSect}
\Gamma(\ovbX,\cal O(l))\twoheadrightarrow \Gamma(D\amalg Z,\cal O(l)),\;
\Gamma(\ovbXp,\cal L(-\pri\Delta)(l))\twoheadrightarrow \Gamma(\pri D\amalg \pri Z,\cal L(-\pri\Delta)(l)).
\end{equation}

Schemes $D$ and $Z$ are finite over the local scheme $U$, hence any line bundle restricted to these schemes is trivial.
%(i.e. has an invertible section), and 
Let's fix an %y 
invertible sections $w\in \Gamma(D,\cal L(l))$, $e\in \Gamma(\cal Z,\cal L(l))$
and fix %also 
$\delta\in \Gamma(\cal L(\pri\Delta))$ such that $Z(\delta)=\pri\Delta$.

Since $Z(\varepsilon^*(e\delta))=\bZp\cap\pri\Delta=Z(\pri\delta|_{\bZp})$,
it follows that the rational section $\varepsilon^*(e\delta)/\pri\delta|_{\bZp}$ is an invertible regular section in $\Gamma(\bZp,\cal L(-\pri\Delta)(l))$.
Using \eqref{eq:lm:et-ex:sur:FinFunc:SurSect} let's choose sections
\begin{equation}\label{eq:lm:et-ex:s:FF:sectLiftDef}\begin{array}{lll}
\pri s\in \Gamma(\ovbX,\cal O(l))\colon\;&
\pri s\big|_{D}=w,\,&
\pri s\big|_{\cal Z}=e\delta,
\\
r\in \Gamma(\ovbXp,\cal L(-\pri\Delta)(l))\colon\;&
r\big|_{\pri D}=\varepsilon^*(w)/\pri\delta\big|_{\pri D} ,\, &
r\big|_{\bZp}=\varepsilon^*(e\delta)/\pri\delta|_{\bZp}
.\end{array}\end{equation}
Then from the above we get
$Z(r)\cap (D\cup \bZp)=\emptyset$,
and since \eqref{eq:lm:et-ex:sur:FF:condDelta} implies that the closed fibre of $\Delta$ is contained in $\bZp$, then
\begin{equation}\label{eq:lm:et-ex:s:FF:Z(r)}Z(r)\cap (D\cup \bZp\cup \pri\Delta)=\emptyset.\end{equation}

Now let's put 
\begin{gather*}
\label{eq:lm:et-ex:s:FF:SectFunDef}
s_0=\varepsilon^*(\pri s)\in \Gamma(\ovbXp,\cal O(l)),\;
s_1=r\delta\in \Gamma(\ovbXp,\cal O(l)),\;
s = (1-t)s_1 +t s_0\in \Gamma(\ovbXp,\cal O(l))\\
\pri f=\pri s/d^l\in k[\cal X],\;
f_0=s_0/d^l\in k[\bXp],\;
f_1=s_1/d^l\in k[\bXp],\,
f=s/d^l\in k[\bXp\times\affl].
\end{gather*}
Then
by definition these functions satisfy the claim of the point 1 of the lemma, and 
by lemma \ref{lm:sect-finfunct} $\pri f$ and $f_0$, $f_1$ and $f$ are relatively finite over $U$ and $\pri U\times\affl$.
Next from \eqref{eq:lm:et-ex:s:FF:Z(r)} we get
$$
Z(f_1)=\pri\Delta\amalg (Z(f_1)-\Delta),\, (\pri Z)\times_{\pri X} (Z(f_1)-\pri\Delta)=\emptyset
$$
that is the claim of the point 4).

So to get the claim of the lemma it is enough to check points 2 and 3.
So the claim follows from that 
$Z(\pri s\big|_{\cal Z})=Z(\delta)$ 
%and $Z(\delta)\times_U (U-Z)\subset 
$Z(s\big|_{\bZp})=\bZp\cap\Delta$ %\subset \bZp\times_{\pri U} Z$.
and from \eqref{eq:lm:et-ex:sur:FF:condDelta}.

%For all $l$ larger some $L$, there is $r\in \Gamma(\cal L(\pri\Delta)^{-1}(l))\colon r\big|_{\pri z}\neq 0$ and such that
%$r\big|_{\pri D}$ is invertible.
%For all $l$ larger some $\pri L$, there is 
%$$\pri s\in \Gamma(\cal O(l))\colon \pri s\big|_{\cal Z}=\varepsilon_*(r\delta\big|_{\bZp}), \pri s\big|_{D}^{-1}.$$
%Set 
%\begin{equation}\label{}
%s_0=\varepsilon^*(\pri s),
%s_1=r\delta,
%s = (1-t)s_1 +t s_0
%\end{equation}
%Then from definitions we get
%\begin{equation}\label{}
%s\big|_{\pri D\times\affl}=r\delta
%s\big|_{\pri Z\times\affl}=r\delta
%\end{equation}
%Define
%$$\pri f=\pri s/d^l\in k[\cal X],
%f=s/d^l\in k[\cal X\times_U \pri U].$$
%By lemma \ref{} $\pri f$ and $f$ are relatively finite over $U$ and $\pri U\times\affl$.
%
%$Z(s_1)=Z(g)\amalg \pri\Delta$
\end{proof}

\section{Zariski excision on affine line}\label{sect:ZarEx} %relative affine line}
\newcommand{\tovr}[2]{ \stackrel{#1}{\,#2^{ \phantom{1} }} } 
\newcommand{\angles}[1]{\langle #1 \rangle}

\begin{theorem}\label{th:afflUzar-ex}
Suppose 
  $\cal F$ is homotopy invariant sheave with GW-transfers over a field $k$
  and $U$ is a local essential smooth scheme ($U= S_s$, $s\in S$, $S\in Sm_k$); 
%  that is 
%  spectrum of the local ring 
%    of a smooth variety $S$ at any point $s$ (over $k$).
then 
  for any Zariski open subscheme $V\subset\affl_U$
    such that $0_U\subset \affl_U$,
  the restriction homomorphism   
    $$i^*\colon \frac{\cal F(\affl_U-0_U)}{\cal F(\affl_U)}\to \frac{\cal F(V-0_U)}{\cal F(V)}$$
  is an isomorphism, 
  where $i\colon V\hookrightarrow\affl_U$ denotes the open immersion.  
\end{theorem}

\begin{proposition}\label{prop:affUzar-ex:inj:GWCor} %label{prop:GWCor-etex}
%For 
%a morphism $i\colon V\hookrightarrow\affl_U$ %, schemes $X$, $X^\prime$, $Z$, $Z^\prime$ and points $z$, $z^\prime$,
%as in theorem \ref{th:afflUzar-ex},
Suppose  
$i\colon V\hookrightarrow\affl_U$ is a morphism as in theorem \ref{th:afflUzar-ex};
then
there is a morphism $\Phi_r\in GWCor( (\affl_U,\affl_U-0_U) , (V,V-0_U) )$
such that $$i \circ \Phi \stackrel{\affl}{\sim} id_{(\affl_U,\affl_U-0_U)} \in GWCor( (\affl_U,\affl_U-0_U) , (\affl_U,\affl_U-0_U) ).$$
\end{proposition}

\begin{proof}

Consider the following divisors in the relative projective line over $\affl_U$:
$$
  T=\prl_{\affl_U}\setminus \affl_{\affl_U},\;\;
  D=\prl_{\affl_U}\setminus (V\times_U {\affl_U}),\;\;
  Z=0\times \affl_U,\;\;
  \Delta = 
   \Gamma(\affl_U\hookrightarrow \prl_U)
   \subset \prl_{\affl\times U}
.$$ 
%i.e. $T$ denotes infinite subscheme in $\prl_{\affl_U}$,
%$D$ is complement closed subscheme in $\prl_{\affl_U}$ to $V\times\affl = V\times_U \affl_U$,
%$\Delta$ is graph over $U$ of immersion $\affl_U\hookrightarrow \prl_U$,
%and $Z$ is zero section over $\affl_U$.
Then let's fix sections 
  $$
  \mu,\nu,\delta\in\Gamma(\prl_{\affl_U},\cal L(1))\colon
  div_0\,\mu = T ,\; 
  div_0\,\nu = 0_{\affl_U} ,\; 
  div_0\,\delta = \Delta,\;  
  \nu\big|_T = \delta\big|_T
  .$$ 
%We can assume in addition that
%  $$\nu\big|_T = \delta\big|_T.$$
%Indeed
%since $T$ is the infinity section of $\prl_{\affl\times U}$ 
%the fraction 
%  $u=\frac{\nu\big|_T}{\delta\big|_T}$ 
%can regarded as the function on $\affl\times U$
%and 
%since 
%intersections of 
%  zero divisors of $\nu$ or $\delta$ 
%  with $T$ 
%  are empty 
%$u$ is invertible function.
%So 
%  if we multiply $\delta$ 
%  by the inverse image of $u$ along the projection 
%  $$\prl_{\affl\times U}\to \affl\times U,$$
%we don't change zero divisor of $\delta$
%and 
%make the required equality holds.

Since $0_U\subset V$ and consequently
$Z\cap D=\emptyset$,
it follows by lemma \ref{lm:corSerreth} that
for a sufficiently large $l$ 
there is a section 
$$\begin{aligned}
s_0\in 
\Gamma(\prl_{\affl_U},\cal L(l)):\;& 
s_0\big|_D=\nu^l\big|_D,\, 
s_0\big|_{Z}=\delta \mu^{l-1} \big|_Z
\\
g\in 
\Gamma(\prl_{\affl_U},\cal L(l-1)):\;& 
g\big|_{\Delta} = \mu^{l-1}\big|_{\Delta},\;
g\big|_{Z}=\delta \mu^{l-1}\big|_Z,\;
g\big|_{T} = \nu^{l-1}
.\end{aligned}$$
Let 
$s = 
s_0\cdot (1-t)+ \delta g\cdot t
\in 
\Gamma(\prl_{\affl_U} \times\affl ,\cal L(l))
.$
Then 
$$
s\big|_{Z\times\affl} = \delta \mu^{l-1},\,
s\big|_{T\times\affl} =  \delta\big|_{T\times\affl}\nu^{l-1}\big|_{T\times\affl}= \nu^l\big|_{T\times\affl}
.$$

By lemma \ref{lm:sect-finfunct} functions $s_0/\mu^l\in k[\affl_{\affl_U}]$ and $s/\mu^l\in k[\affl_{\affl_U\times\affl}]$
are relatively finite, then we can apply construction from proposition \ref{prop:constrOrCurFQCor} and put
$$
Q_0=\langle dy,s_0/\mu^l\rangle=(k[Z_0],q_0)\in Q(\cal P(Z_0\to \affl_U)),\;
Q=\langle dy,s/\mu^l\rangle=(k[Z],q)\in Q(\cal P(Z\to \affl_U\times\affl))
,$$
where $$Z_0 = Z(s_0)\subset \affl_{\affl_U},\quad Z = Z(s)\subset \affl_{\affl_U\times\affl}.$$

Again since
$0_U\subset V$,
it follows that
$s_0\big|_{D}=\nu^l\big|_{D}$ is invertible, 
and hence 
$ Z_0 \subset V\times_U \affl_U.$
Then since $s\big|_{T\times\affl}=\nu^l\big|_{T\times\affl}$ is invertible,
it follows that
$Z \subset \affl_U\times_U \affl_U.$
Let's denote by
$$i_{Z_0}\colon Z_0\to V\times_U \affl_U,\;
i_{Z}\colon Z\hookrightarrow \affl_U\times_U\times\affl_U=\affl_{\aff_U}$$
the canonical closed injections. 

Next since
$Z(\delta)=\Delta$ and so $\delta$ is invertible on $0_U\times_U (V-0_U)=0\times (V-0_U)$ 
%(in sense of identification \eqref{eq:affzar-ex:CoordCorCurve}),
and consequently $s_0$ and $s$ are invertible on 
$0\times (\affl_U-0_U)$ and $0\times (\affl_U-0_U)\times\affl$ respectively,
it follows that
$$
0_U\times_{{V}} Z_0\times_{\affl} (\affl-0_U)  = \emptyset,\quad
0\times_{{\affl}} Z\times_{{V}} (\aff_U-0_U)  = \emptyset.
$$ 
%$$
%0_U\times_{{V}} Z_0\times_{\stackrel{(2)}{\affl}} (\affl-0_U)  = \emptyset,\quad
%0\times_{{\affl}} Z\times_{\stackrel{(2)}{V}} (V-0_U)  = \emptyset.
%$$ 
Hence by lemma \ref{lm:QPairCor} 
the quadratic spaces 
$i_{Z_0} \circ Q_0$ and $i_Z\circ Q$
defines GW-correspondences between pairs
\begin{gather*}
\widetilde{\Phi}=[i_{Z_0} \circ Q_0]\in GWCor( (\affl_U,\affl_U-0_U) , (V,V-0_U) ),\\
\widetilde{\Theta}=[i_Z\circ Q]\in GWCor( (\affl_U\times\affl,(\affl_U-0_U)\times\affl) , (\affl_U,\affl_U-0_U) )
.\end{gather*}
%\begin{gather*}
%\xymatrix{
%\stackrel{ dY,f }{ \affl_V }\ar[r] \ar[d] & \stackrel{ dY,h }{ \affl_{V\times\affl} } \ar[d] & \stackrel{ dY, (X-Y)^{2n+1} }{ \affl_V } \ar[l] \ar[d] \\ 
%V\ar[r] & V\times\affl & V\ar[l]
%\\
%U \ar[r] & V & V\ar@{=}[l] \\
%\stackrel{ (k[Z_0],q_{0}) }{ Z_0 }\ar[d]\ar[u] \ar[r] &
%\stackrel{ (k[Z],q) }{ Z\ar[d]\ar[u] } &
%\stackrel{ (k[Z_1],q_1) }{ Z_1 }\ar[d]\ar[u] \ar[l]\\
%V\ar[r]^{(0)_V} & V\times\affl & V\ar[l]_{(1)_V}
%\\
%U \ar[r] & V & V\ar@{=}[l] \\
%\\
%V\ar[uu]|{\Phi}\ar[r] & V\times\affl\ar[uu]|{\Theta} &V\ar[uu]_{\sim Id_V}\ar[l] 
%}
%\end{gather*} 
Thus
$$
i \circ \widetilde{\Phi}=
[i\circ i_{Z_0}\circ \langle dy,s_0/\mu^l \rangle]=
[i\circ i_{Z_0}\circ i_0^*(\langle dy,s/\mu^l \rangle)]=
[i_{Z}\circ \langle dy, s/\mu^l \rangle\circ i_0]=
\widetilde{\Theta}\circ i_0
.$$

On other side, since $g\big|_{\Delta}=\mu^{l-1}$ then $Z(\delta g)=\Delta\amalg Z(g)$,
and since $g\big|_{Z}=\mu^{l-1}$ then $Z(g)\subset \affl_U - 0_U$,
using lemma \ref{lm:QPairCor} we get that
$$
\widetilde{\Theta}\circ i_1=
[i_{Z(\delta g)}\circ \langle dy,\delta g \rangle]=
[(k[\Delta],u)]
\in GWCor( (\affl_U,\affl_U-0_U) , (\affl_U,\affl_U-0_U) )
,$$
for some invertible $u\in k[\affl_U]^*$.
Thus if we put 
$$
\Phi=\widetilde{\Phi}\circ [\langle u^{-1}\rangle],\;
\Theta=\widetilde{\Theta}\circ [\langle u^{-1}\rangle],\;
,$$
then 
$$\Theta\circ i_0=i \circ \Phi, \Theta\circ i_1 = [(k[\Delta],1)]=id_{(\affl_U,\affl_U-0_U)},$$
and so $i \circ \Phi\stackrel{\affl}{\sim} =id_{(\affl_U,\affl_U-0_U)}$. % by homotopy $\Theta$.

\end{proof}

\begin{proposition}\label{prop:affUzar-ex:sur:GWCor} %label{prop:GWCor-etex}
Suppose  
$i\colon V\hookrightarrow\affl_U$ is a morphism as in theorem \ref{th:afflUzar-ex};
then 
there is a morphism $\Phi_l\in GWCor( (\affl_U,\affl_U-0_U) , (V,V-0_U) )$
such that $$\Phi \circ i \stackrel{\affl}{\sim} id_{(V,V-0_U)} \in GWCor( (V,V-0_U) , (V,V-0_U) ).$$
\end{proposition}\begin{proof}

Similarly as in the proof of proposition \ref{prop:affUzar-ex:inj:GWCor}
let's denote :
\begin{multline*}
  T=\prl_{\affl_U}\setminus \affl_{\affl_U},\;
  T^\prime=\prl_{V}\setminus \affl_{V},\\
  D=(\prl_{U}\setminus V)\times_U {\affl_U}\subset \prl_{\affl_U},\;
  Z=0\times \affl_U\subset \prl_{\affl_U},\;
  D^\prime = D\times_{\affl_U} V \subset \prl_{V},\;
  Z^\prime=0\times V\subset \prl_V,\\
  \Delta = 
   \Gamma(\affl_U\hookrightarrow \prl_U)
   \subset \prl_{\affl\times U},\;
  \Delta^\prime = 
   \Gamma(V\hookrightarrow \prl_U)
   \subset \prl_{V}
,\end{multline*} 
and fix some sections
  $$
  \mu,\nu,\delta\in\Gamma(\prl_{\affl_U},\cal L(1))\colon
  div_0\,\mu = T ,\; 
  div_0\,\nu = Z ,\; 
  div_0\,\delta = \Delta ,\; 
  \nu\big|_T = \delta\big|_T
  .$$
Also let's denote by $v\colon \prl_{V}\hookrightarrow \prl_{\affl_U}$ the immersion defined by the
base change.   

Since $D\cap V =\emptyset$,
and consequently $\Delta\cap D^\prime = \emptyset$, 
then $\delta$ is invertible on $D_V$.  
Let's denote the inverse section by $\delta^{-1} \in \Gamma(\cal L(-1)_{D^\prime})$.
Next by lemma \ref{lm:corSerreth} 
  for sufficiently large  $l$
  there exist sections 
   $$\begin{aligned}
  s^\prime\in \Gamma(\cal L(l),\prl_{\affl_U})\colon \;&
     \pri s\big|_{D}=\nu^l,\;&&
         \pri s\big|_{Z}=\mu^{l-1}\cdot\delta
  &\\
   g\in \Gamma(\cal L((l-1)),\prl\times V)\colon \;&
     g\big|_{D_V}   =\nu^l\cdot\delta^{-1},\;&
       g\big|_{\Delta}   =\mu^{n-1},\;&
         g\big|_{0_U\times V}=\mu^{n-1}
   &,\end{aligned}$$
%because 
%  intersection 
%    of $D_{\affl_U}$ 
%    with $0_{\affl_U}$ 
%    into $\prl_{\affl_U}$ 
%  is empty
%  and
%  intersection 
%    of $D_V$ 
%    with $\Delta$ 
%    into $\prl_V$ 
%  is empty too.  

Then we can define 
  sections 
  $$\begin{aligned}
  &s_0 & \in \;&
      \Gamma(\cal L(l),\prl\times{V})
    \colon\quad  &
    s_0 = v^*(\pri s) &\\
  &s_1 & \in \;& 
      \Gamma(\cal L(l),\prl_{V})
    \colon\quad  &
    s_1 = g\cdot\delta &\\
  &s  & \in \;& 
      \Gamma(\cal L(l),\prl_{V\times\affl})
    \colon\quad &
    s=s_0\cdot (1-t)+s_1\cdot t
  &.\end{aligned}$$ 

Then %by definition of $s^\prime$, $g$, $s_0$, $s_1$ and $s$
  $$
s\big|_{\pri D\times\affl} =
  s_0\big|_{\pri D} =
     s_1\big|_{\pri D} =
\nu^l\big|_{\pri D},\; 
s\big|_{Z^\prime\times\affl} = 
  s_0\big|_{Z^\prime} =
     s_1\big|_{Z^\prime} =
\mu^{l-1}\cdot\delta ,\;
%s_1\big|_{\Delta}=0 ,\; 
  div\,s_1 = div\,g \coprod  \Delta
.$$ 

%Since 
%  $s$ is invertible on $D_{V\times\affl}$ and 
%  $s_0$ is invertible on $D_{\affl_U}$, 
% $$
% Z_0 = Z(s_0) \subset V\times_U\affl_U,\;
% Z = Z(s) \subset V\times_U V  \times\affl
% .$$
%So we can put 

%Then 
%$$
%s\big|_{Z\times\affl} = \delta^l,\,
%s\big|_{T\times\affl} = \nu^l\big|_{T\times\affl} = \delta^n\big|_{T\times\affl}
%.$$

By lemma \ref{lm:sect-finfunct} functions $s^\prime/\nu^l\in k[\affl_{\affl_U}]$ and $s/\nu^l\in k[\affl_{V\times\affl}]$
are relatively finite, then we can apply construction from proposition \ref{prop:constrOrCurFQCor} and put
$$
Q^\prime=\langle dy,s_0/\nu^l\rangle=(k[Z^\prime],q^\prime)\in Q(\cal P(Z^\prime\to \affl_U)),\;
Q=\langle dy,s/\nu^l\rangle=(k[Z],q)\in Q(\cal P(Z\to V\times\affl))
,$$
where $Z^\prime = Z(s^\prime)\subset \affl_{\affl_U},\quad Z = Z(s)\subset \affl_{V\times\affl}.$

Since
$0_U\subset V$,
it follows that
$s^\prime\big|_{D}=\nu^l\big|_{D}$ is invertible, 
and consequently 
$ Z^\prime \subset V\times_U \affl_U.$
Since $s\big|_{T\times\affl}=\nu^l\big|_{T\times\affl}$ is invertible,
it follow that
$Z \subset V\times_U V$.
Let's denote by
$$
i_{Z^\prime}\colon Z^\prime\to V\times_U \affl_U,\;
i_{Z}\colon Z\hookrightarrow V\times_U\times V\times\affl,\;
i_{Z_0}\colon Z^\prime=Z(s_0)\hookrightarrow V\times_U\times V,\;
i_{Z_1}\colon Z(s_1)\hookrightarrow V\times_U\times V
$$
the canonical closed injections. 

Next since
$Z(\delta)=\Delta$ and so $\delta$ is invertible on $0_U\times_U (V-0_U)=0\times (V-0_U)$ 
%(in sense of identification \eqref{eq:affzar-ex:CoordCorCurve}),
and consequently $s_0$ and $s$ are invertible on 
$0\times (\affl_U-0_U)$ and $0\times (V-0_U)\times\affl$ respectively,
it follows that
$$
0_U\times_{{V}} Z^\prime\times_{\affl} (\affl-0_U)  = \emptyset,\quad
0\times_{{\affl}} Z\times_{{V}} (V-0_U)  = \emptyset.
$$ 
%$$
%0_U\times_{{V}} Z_0\times_{\stackrel{(2)}{\affl}} (\affl-0_U)  = \emptyset,\quad
%0\times_{{\affl}} Z\times_{\stackrel{(2)}{V}} (V-0_U)  = \emptyset.
%$$ 
Hence by lemma \ref{lm:QPairCor} 
the quadratic spaces 
$i_{Z^\prime} \circ Q^\prime$ and $i_Z\circ Q$
define GW-correspondences between pairs
\begin{gather*}
\widetilde{\Phi}=[i_{Z^\prime} \circ Q^\prime]\in GWCor( (\affl_U,\affl_U-0_U) , (V,V-0_U) ),\\
\widetilde{\Theta}=[i_Z\circ Q]\in GWCor( ((V\times\affl,(V-0_U)\times\affl) , (V,V-0_U) )
.\end{gather*}

%\begin{gather*}
%\xymatrix{
%\stackrel{ dY,f }{ \affl_V }\ar[r] \ar[d] & 
%\stackrel{ dY,h_0 }{ \affl_U }\ar[r] \ar[d] & 
%\stackrel{ dY,h }{ \affl_{U\times\affl} } \ar[d] & 
%\stackrel{ dY, (X-Y)\cdot g }{ \affl_U } \ar[l] \ar[d] \\ 
%V & U\ar[l]\ar[r] & U\times\affl & U\ar[l]
%\\
%& & U &  \\
%\stackrel{ (k[Z^\prime],q^\prime) }{ Z^\prime }\ar[d]\ar[urr] & 
%\stackrel{ (k[Z_0],q_{0}) }{ Z_0 }\ar[d]\ar[ur] \ar[r]\ar[l] & 
%\stackrel{ (k[Z],q) }{ Z\ar[d]\ar[u] } &
%\stackrel{ (k[Z_1],q_1) }{ Z_1 }\ar[d]\ar[ul] \ar[l]\\
%V & U \ar[r]\ar[l] & U\times\affl   & U  \ar[l] 
%\\
% & & U & \\
%\\
%V\ar[uurr]|{\Phi} & U\ar[uur]|{\Theta_0} \ar[r]\ar[l] & U\times\affl \ar[uu]|{\Theta}  & U \ar[uul]|{\Theta_1}_{\sim Id_U} \ar[l] 
%}
%\end{gather*} 

Then
$$
\widetilde{\Phi}\circ i=
[i_{Z^\prime}\circ \langle dy,i^*(s^\prime/\nu^l) \rangle]=
[i_{Z_0}\circ \langle dy,s_0/\nu^l \rangle]=
[i_{Z_0}\circ \langle dy, i_0^*(s/\nu^l) \rangle]=
[i_{Z}\circ \langle dy, s/\nu^l \rangle\circ i_0]=
\widetilde{\Theta}\circ i_0
$$
and using lemma \ref{lm:QPairCor} we get that
$$
\widetilde{\Theta}\circ i_1=
[i_{Z_1}\circ \langle dy,s_1/\nu^l \rangle]=
[(k[\Delta^\prime],u^\prime)]+[(k[Z_1-\Delta^\prime],q_1)]=
[(k[\Delta^\prime],u^\prime)]
,$$
for some invertible $u^\prime\in k[\affl_U]^*$.
Thus if we put 
$$
\Phi=[\langle u^{-1}\rangle]\circ \widetilde{\Phi},\;
\Theta=[\langle u^{-1}\rangle]\circ \widetilde{\Theta}
,$$
then 
$\Theta\circ i_0=\Phi\circ i$, $\Theta\circ i_1 = [(k[\Delta^\prime],1)]=id_{(V,V-0_U)},$
and so $i \circ \Phi\stackrel{\affl}{\sim} =id_{(V,V-0_U)}$. % by homotopy $\Theta$.

\end{proof}
\begin{proof}[Proof of the theorem \ref{th:afflUzar-ex}.]
As noted in remark \ref{rm:GWCorFpair},
for homotopy invariant presheave with GW-transfers $\cal F$,
formula 
$$\begin{array}{ccl}
GWCor^{pair}&\longrightarrow& Ab\\
(Y,U)&\mapsto& \Coker(\cal F(Y)\to \cal F(U))
\end{array}$$
defines homotopy invariant presheave on the category $GWCor^{pair}$.
Hence the injectivity of the homomorphism $i^*$ follows from proposition \ref{prop:affUzar-ex:inj:GWCor}, and
the surjectivity follows from proposition \ref{prop:affUzar-ex:sur:GWCor}.

\end{proof}

\begin{theorem}\label{th:afflzar-ex}
Let 
  $\cal F$ be a homotopy invariant presheave with GW-transfers over field $k$
  and $K$ be a geometric extension $K/k$.
Then 
  for any Zariski open subschemes $U\subset V\subset\affl_K$ and point $z\in U$,
  the restriction homomorphism   
    $$i^*\colon {\cal F(V-z)/\cal F(V)}\to {\cal F(U-z)/\cal F(U)}$$
  is an isomorphism, 
  where $i\colon U\hookrightarrow V$ denotes the open immersion.  
\end{theorem}\begin{proof}
The proof is similar to the proof of theorem \ref{th:afflzar-ex},
but we should change the section $$\nu\in \Gamma(\prl_{\affl_K},\cal O(1))\colon div_0\,\nu=0\times{\affl_K}$$ 
in the construction of the inverse morphisms by a section of the line bundle with zero divisor being any divisor of the form 
$x\times\affl_K$, where $x\in V$ is a rational point over $K$. %, and $x\in V$. 
Such point always exists is $k$ is infinite.
\end{proof}
\begin{corollary}\label{cor:afflzar-ex}
Let 
  $\cal F$ is be a homotopy invariant presheave with GW-transfers over field $k$
  and $K$ be a geometric extension $K/k$ %(i.e. field of functions of some variety).
Then 
  for any Zariski open subschemes $U\subset V\subset\affl_K$ and point $z\in U$,
  the restriction homomorphism   
    $$i^*\colon {\cal F(V-z)/\cal F(V)}\to {\cal F(U-z)/\cal F(U)}$$
  is an isomorphism, 
  where $i\colon U\hookrightarrow V$ denotes the open immersion.  
\end{corollary}
\begin{proof}
The claim follows by isomorphisms 
${\cal F(V-z)/\cal F(V)}\simeq {\cal F(\affl_K-z)/\cal F(\affl_K)}\simeq {\cal F(U-z)/\cal F(U)}$
\end{proof}

%\begin{proposition}\label{prop:affUzar-ex:inj:GWCor} %label{prop:GWCor-etex}
%For 
%a morphism $i\colon U\hookrightarrow V$ and point $z$ %, schemes $X$, $X^\prime$, $Z$, $Z^\prime$ and points $z$, $z^\prime$,
%as in theorem \ref{th:afflzar-ex},
%there is
%morphism $\Phi_r\in GWCor( (V,V-z) , (U,U-z) )$,
%such that $$i \circ \Phi \stackrel{\affl}{\sim} id_{(V,V-z)} \in GWCor( (V,V-z) , (V,V-z) ).$$
%
%\end{proposition}

\section{Etale excision}\label{sect:EtEx}

\begin{theorem}\label{th:et-ex}%[etale excision, theorem \ref{}]
Let 
  $\cal F$ be a homotopy invariant presheave with GW-transfers
and 
  $\pi\colon  X^\prime\to X$ be an etale morphism of smooth schemes over a geometric extension $K/k$;
let 
  $Z\subset X$ be a reduced closed subscheme of codimension one %1 
    such that  
    $\pi$ induces the isomorphism of 
      $Z$ and the preimage  $Z^\prime=\pi^{-1}(Z)$;
and let
$z\in Z$ and $z^prime\in \pri Z$ are points such that $\pi(z^\prime)-z$.
%  $z$ be a point in $Z$ and $z^\prime$ be it's preimage in $\pri Z$.

Then $\pi$ induces the isomorphism  %the 
   $$\pi^*\colon \frac{\cal F(X_z-Z_z)}{\cal F(X_z)}  
              \stackrel{\sim}{\to}  
                     \frac{\cal F(X^\prime_{z^\prime}-Z^\prime_{z^\prime})}{\cal F(X^\prime_{z^\prime})}.$$
\end{theorem}

\begin{proposition}\label{prop:et-ex:inj:GWCor} %label{prop:GWCor-etex}
Suppose $\pi$, $X$, $X^\prime$, $Z$, $Z^\prime$, $z$, $z^\prime$,
as in theorem \ref{th:et-ex};
then there is
a morphism $\Phi_l\in GWCor( (X,Z)_z , (X^\prime,Z^\prime) )$
such that $\pi \circ \Phi \stackrel{\affl}{\sim} i_z \in GWCor( (X,Z)_z , (X,Z) ) $,
where $i_z\colon (X,Z)_z\hookrightarrow (X,Z) $ denotes the canonical morphism of pairs
and $\pi$ is considered as a morphism $\pi\colon (\pri X,\pri Z) \to (X,Z) $.
\end{proposition}
\begin{proof}%[Proof of proposition \ref{prop:et-ex:inj:GWCor}]

\newcommand{\mcX}{\mathcal{X}}
\newcommand{\mcXp}{\pri{\mathcal{X}}}
\newcommand{\mcXpp}{\ppri{\mathcal{X}}}
\newcommand{\mcZ}{\mathcal{Z}}
\newcommand{\mcZp}{\pri{\mathcal{Z}}}
\newcommand{\mcZpp}{\ppri{\mathcal{Z}}}

%By assumption of the theorem in terms of def. \ref{def:nisneigh} we have 
In terms of def. \ref{def:nisneigh} the assumptions of the theorem \ref{th:et-ex} give us
a Nisnevich neighbourhood $(X^\prime,Z^\prime)\to (X,Z)$,
such that $codim\,Z^\prime=codim\,Z=1$ and points $z^\prime\in X^\prime$ and $z\in X$, $\pi(z^\prime)=z$.
Using lemma \ref{lm:et-ex:relCures} we modify it
to a (relative) Nisnevich neighbourhood
$\varpi\colon (\tXp,{\tilde Z}^\prime)\to (\tilde X,\tilde Z)$
over some essential smooth local base $S$
equipped with a good relative compactification
$\overline{\tilde{\pi} }\colon \ovtXp\to \ovtX$.
Lemma \ref{lm:et-ex:relCures} implies in addition
that $\bXp$ and $\cal X$ are smooth over $S$,
there are trivialisations of relative canonical classes 
$\mu^\prime\colon \omega_S(\bXp)\simeq \cal O(\bXp)$ and $\mu\colon \omega_S(\cal X)\simeq \cal O(\cal X)$,
and 
that there is a very ample bundle
$\cal O(1)$ on $\ovtX$,
such that 
$\ovvarpi^*(\cal O(1))$ is very ample too.
To shortify notations 
let's denote $\ovvarpi^*(\cal O(1))$ by the same symbol $\cal O(1)$.
$$\xymatrix{
X^\prime \ar[d]^{\pi} & \tXp\ar[d]^{\tilde \pi}\ar[l]^{\pri v}\ar[r]^{\pri j} & \ovtXp\ar[d]^{\ovtpi} &\\
X & \tilde X\ar[l]^v\ar[r]^j & \ovtX \ar[r] & S\\
%\tXp\ar[r]^{\varpi}\ar[u]^{\pri v} & \tilde X\ar[u]^v
}$$

Denote $U=X_z$ and consider the base change along $X_z\to S$ (see the first digram of \eqref{eq:et-ex:inj:basechange}).
Set 
$(\cal X,\cal Z) = (\tilde X,\tilde Z)\times_S U$,
$\ovbX = \ovtX\times_S U$,
$(\bXp,\bZp) = (\tXp,\tZp)\times_S U$,
$\ovbXp = \ovtXp\times_S U$.
(note that $\cal Z=Z\times_S U$, $\bZp=\pri Z\times_S U$).

Denote by $\Delta$ the graph of the canonical embedding 
$U=X_z=\tilde X_z\hookrightarrow \tilde X$ considered as a closed subset in $\ovbX$,
and let $\Delta^\prime = \varpi\big|_{\bZp}^{-1}( i_{\cal Z}^{-1}(\Delta) )$, 
where $i_{\cal Z}\colon \cal Z\to \cal X$ denotes the canonical injection
(so $\Delta^\prime \subset \bXp$ is the image of $\Delta\cap \cal Z$ after the identification $\bZp\simeq \cal Z$).
%Also we denote $D = \ovbX\setminus \cal X$, $D^\prime = \ovbXp\setminus \bXp$.

Then $\cal Z\cap \Delta$ is equal to the diagonal in $Z_z\times_S Z_z$ and so $\cal Z\times_{\bX} \Delta\times_U (U-Z)=\emptyset$.
Hence we can apply lemma \ref{lm:et-ex:inj:FinFuncSup}
and find relatively finite (over $X_z$) regular functions 
$f_0,f_1\in k[\cal X]$ and $f\in k[\cal X\times\affl]$
%that are 
(see def. \ref{def:OrCurFinFun})
such that 
%\begin{itemize}[leftmargin=15pt]
%\item[1)]
%$i_0^*(f)=f_0, i_1^*(f)=f_1$, 
%where $i_0,i_1\colon U\to U\times\affl$ denote zero and unit sections;
%\item[2)]
%there is a lift $l\colon Z(f_0)\to \bXp$
%of the canonical injection $i_{Z(f_0)}\colon Z(f_0)\hookrightarrow \bX$;
%\item[3)] and the following conditions holds:
%\begin{gather}
%Z(f)\times_{U\times\affl }(U\times \affl-Z_z\times\affl)=(X-Z) \times_X Z(f)\times_{U\times\affl} (U\times \affl-Z_z\times\affl)
%\label{eq:et-ex:Zfpair}\\
%Z(f_1)=\Delta\amalg (Z(f_1)-\Delta),\, (Z(f_1)-\Delta)=(X-Z)\times_{X} (Z(f_1)-\Delta)
%\label{eq:et-ex:Z1}\\\label{eq:et-ex:lZpair}
%l(Z(f_0))\times_U (U-Z_z)=(X^\prime-Z^\prime) \times_{X^\prime} l(Z(f_0))\times_U (U-Z_z)
%\end{gather}
%%\begin{equation}\label{eq:et-ex:lZpair}
%%l(Z(f_0))\times_U (U-Z_z)=(X^\prime-Z^\prime) \times_{X^\prime} l(Z(f_0))\times_U (U-Z_z)
%%.\end{equation}
%\end{itemize}
\begin{itemize}[leftmargin=15pt]
\item[1)]
$i_0^*(f)=f_0, i_1^*(f)=f_1$, 
where $i_0,i_1\colon U\to U\times\affl$ denote zero and unit sections;
\item[2)] and the following conditions holds:
\begin{gather}
Z(f)\times_{U\times\affl }(U\times \affl-Z_z\times\affl)=(X-Z) \times_X Z(f)\times_{U\times\affl} (U\times \affl-Z_z\times\affl)
\label{eq:et-ex:Zfpair}\\
Z(f_1)=\Delta\amalg (Z(f_1)-\Delta),\, (Z(f_1)-\Delta)=(X-Z)\times_{X} (Z(f_1)-\Delta)
\end{gather}
\item[3)]
there is a lift $l\colon Z(f_0)\to \bXp$
of the canonical injection $i_{Z(f_0)}\colon Z(f_0)\hookrightarrow \bX$ 
such that
\begin{equation}\label{eq:et-ex:lZpair}
l(Z(f_0))\times_U (U-Z_z)=(X^\prime-Z^\prime) \times_{X^\prime} l(Z(f_0))\times_U (U-Z_z)
.\end{equation}
\end{itemize}
%
%such that $i_0^*(f)=f_0, i_1^*(f)=f_1$, 
%where $i_0,i_1\colon U\to U\times\affl$ denote zero and unit sections,
%and such that
%\begin{gather}
%Z(f)\times_{U\times\affl }(U\times \affl-Z_z\times\affl)=(X-Z) \times_X Z(f)\times_{U\times\affl} (U\times \affl-Z_z\times\affl)
%\label{eq:et-ex:Zfpair}\\
%Z(f_1)=\Delta\amalg (Z(f_1)-\Delta),\, (Z(f_1)-\Delta)=(X-Z)\times_{X} (Z(f_1)-\Delta)
%\label{eq:et-ex:Z1}
%\end{gather}
%and regular map
%$$
%l\colon Z(f_0)\to \bXp\colon 
%\varpi \circ l=i_{Z(f_0)},$$
%where
%$i_{Z(f_0)}\colon Z(f_0)\hookrightarrow \bX$ is cacnonical injection, and such that
%\begin{equation}\label{eq:et-ex:lZpair}
%l(Z(f_0))\times_U (U-Z_z)=(X^\prime-Z^\prime) \times_{X^\prime} l(Z(f_0))\times_U (U-Z_z)
%.\end{equation}

The inverse images of the trivialisation $\mu$ define the trivialisations 
$\mu_{X_z}\colon \omega_{X_z}(\cal X)\simeq \cal O(\cal X)$, 
$\mu_{X_z\times\affl}\colon \omega_{X_z\times\affl}(\cal X\times\affl) \simeq \cal O(\cal X\times\affl),$
%$$\mu_{X_z}\colon \omega_{X_z}(\cal X)\simeq \cal O(\cal X),\, 
%\mu_{X_z\times\affl}\colon \omega_{X_z\times\affl}(\cal X\times\affl) \simeq \cal O(\cal X\times\affl),$$
%So we get oriented curve $(\cal X,\mu_{X_z})$ over $X_z$.
so the base changes along the zero and unit sections of $X_z\times\affl$ give us morphisms of oriented relative curves with relative finite functions (see the following diagram and see def. \ref{def:OrCurFinFun} for the notion of oriented curve with relative finite functions).
Then applying the construction from proposition \ref{prop:constrOrCurFQCor} %consequently 
%and composition of produced quadratic spaces with the injections $v$ and $v^\prime$, %regular maps $v$ and $v^\prime$, 
we get the quadratic spaces \begin{multline*}
(k[Z(f_0)],q_0) = \langle \bX, \mu_{X_z}, f_0  \rangle \in Q(\cal P(Z(f_0)\to X_z)),\\
(k[Z(f_1)],q_1) = \langle \bX, \mu_{X_z}, f_1  \rangle \in Q(\cal P(Z(f_1)\to X_z)),\\
(k[Z(f)],q) = \langle \bX, \mu_{X_z\times\affl}, f  \rangle \in Q(\cal P(Z(f)\to X_z\times\affl)),
\end{multline*}
%$(k[Z_0],q_0)$, $(k[Z],q)$, $(k[Z_1],q_1)$. 
%Lemma \ref{lm:QPairCor} and equations \eqref{eq:et-ex:Zfpair} and \eqref{eq:et-ex:lZpair} 
%implies that 
%the spaces $(k[Z_0],q_0)$, $(k[Z],q)$, $(k[Z_1],q_1)$ define GW-correspondences $\Phi$, $\Theta_0$, $\Theta$, $\Theta_1$ between pairs. % (see the following diagram).
%%$\Phi\in GWCor((X_z,X_z-Z_z), (X^\prime,X^\prime-Z^\prime))$, $\Theta_0\colon GWCor( (X_z,X_z-Z_z), (X,X-Z))$, $\Theta\colon GWCor( (X_z,X_z-Z_z) \times\affl , (X,X-Z) )$, $\Theta_1\colon GWCor( (X_z,X_z-Z_z), (X,X-Z))$.
%%Finally, commutativity of the diagram implies that $\Phi$, $\Theta_0$, $\Theta$, $\Theta_1$
%% this GW-correspondences 
%satisfy the required relations up to twist by invertible function $u\in k[U]^*$:
\begin{equation*}%\label{eq:et-ex:inj:diagrams}
\xymatrix{ &
\stackrel{(\mu_{X_z} , f_0 )}{ \bX } \ar[r]_{(0)_\bX} \ar[d] & 
\stackrel{(\mu_{X_z\times\affl},  f )}{ \bX\times\affl } \ar[d] & 
\stackrel{(\mu_{X_z}, f_1 )}{ \bX } \ar[l]^{(1)_\bX} \ar[d]
  \\ &
 X_z\ar[r]^{(0)_{X_z}} & {X_z}\times\affl & {X_z}\ar[l]_{(1)_{X_z}}
\\
X^\prime \ar[rr]^{\pi} && X &\\
&
\stackrel{ (k[Z_0],q_{0}) }{Z_0} \ar[ul]\ar[ur] \ar[d] \ar[r] & 
\stackrel{ (k[Z],q)  }{Z}\ar[u] \ar[d] & 
\stackrel{ (k[Z_1],q_{1} ) }{Z_1} \ar[ul] \ar[d] \ar[l]
  \\
&{X_z}\ar[r]_{(0)_{X_z}} & {X_z}\times\affl & {X_z}\ar[l]^{(1)_{X_z}}
\\
(X^\prime,X^\prime-Z^\prime) \ar[rr]^{\pi} && (X,X-Z) &
  \\ \\
&{(X_z,X_z-Z_z)}\ar[r]_{(0)_{X_z}} \ar[uul]|{\Phi}\ar[uur]|{\Theta_0} & {(X_z,X_z-Z_z)}\times\affl \ar[uu]|{\Theta} & {(X_z,X_z-Z_z)}\ar[l]^{(1)_{(X_z,X_z-Z_z)}} \ar[luu]|{\Theta_1}_{\sim \text{embedding  } {(X_z,X_z-Z_z)}\hookrightarrow (X,X-Z)}
}
\end{equation*}
%This means that we define 
%So we have define \begin{multline*}
%(k[Z(f_0)],q_0) = \langle \bX, \mu_{X_z}, f_0  \rangle \in Q(\cal P(Z(f_0)\to X_z)),\\
%(k[Z(f_1)],q_1) = \langle \bX, \mu_{X_z}, f_1  \rangle \in Q(\cal P(Z(f_1)\to X_z)),\\
%(k[Z(f)],q) = \langle \bX, \mu_{X_z\times\affl}, f  \rangle \in Q(\cal P(Z(f)\to X_z\times\affl)),
%\end{multline*}
%$$\xymatrix{
%}
%.$$

Now consider the compositions with the morphisms of the base change diagram
\begin{equation}\label{eq:et-ex:inj:basechange}
\xymatrix{
X&\tilde X\ar[d]\ar[l]_{v} & \bX\ar[l]_{e}\ar[d] & \bX\times\affl \ar[l]_{pr_{\affl}}\ar[d] \\
&S & U\ar[l] &U\times\affl\ar[l]
}\quad\xymatrix{
\pri X&\tXp\ar[d]\ar[l]_{\pri v} & \bXp\ar[l]_{\pri e}\ar[d]  \\
&S & U\ar[l] &
}\end{equation}
and canonical injection 
$i_{Z(f)}\colon Z(f)\hookrightarrow \bX\times\affl$.
Using lemma \ref{lm:QPairCor} and equalities \eqref{eq:et-ex:Zfpair} and \eqref{eq:et-ex:lZpair} we can put 
$$\begin{array}{ccccc}%\begin{gather*}
\Phi &=& [v^\prime \circ \pri e \circ l \circ (k[Z(f_0)],q_0)] &\in& GWCor( (X_z,X_z-Z_z) , (X^\prime,X^\prime-Z^\prime) ),\\
\Theta &=& [v \circ e \circ pr_{\affl} \circ i_{Z}\circ (k[Z],q)] &\in& GWCor( (X_z,X_z-Z_z)\times\affl, (X,X-Z)),
%\Phi = [v^\prime \circ \pri e \circ l \circ (k[Z(f_0)],q_0)] \in GWCor( (X_z,X_z-Z_z) , (X^\prime,X^\prime-Z^\prime) ),\\
%%\Phi_0 = [v \circ i_{Z}\circ (k[Z_0],q_0)],
%\Theta = [v \circ e \circ pr_{\affl} \circ i_{Z}\circ (k[Z],q)]\in GWCor( (X_z,X_z-Z_z)\times\affl, (X,X-Z)),
%%\Phi_1 = [v \circ i_{Z}\circ (k[Z_1],q_1)],
\end{array}$$%\end{gather*}
and get
\begin{multline*}
\pi \circ \Phi = [\pi\circ v^\prime \circ l\circ (k[Z_0],q_0)] = 
[v\circ \varpi \circ l\circ (k[Z(f_0)],q_0)] = \\ =
[v \circ i_{Z(f_0)}\circ (k[Z(f_0)],q_0)] = 
[v \circ i_{Z(f)}\circ (k[Z],q)\circ i_0] = 
\Theta\circ i_0,
\end{multline*}and
\begin{multline*}
\Theta\circ i_1=
[v \circ i_{Z(f_1)}\circ (k[Z(f_1)],q_1)] = 
[v \circ i_{\Delta}\circ (k[\Delta],u)] + 
[v \circ i_{Z(f_1)-\Delta}\circ (k[Z(f_1)-\Delta],q_1\big|_{Z(f_1)-\Delta})] = \\ = 
[v \circ i_{\Delta}\circ (k[\Delta],u)]=
[i_{X_z}\circ \langle u\rangle]\in GWCor( (X_z,X_z-Z_z), (X,X-Z) )
\end{multline*}
where the second and third equalities follows from lemma \ref{lm:QPairCor} and \eqref{eq:et-ex:Zfpair}. %?? and \eqref{eq:et-ex:Z1}.
Thus if we replace $\Phi$ and $\Theta$ by $\Phi\circ \langle u^{-1} \rangle$
and $\Theta\circ \langle u^{-1} \rangle$.
the claim follows.

\end{proof}
%\begin{corollary}\label{cor:et-ex:inj}
%Homomorphism $\pi^*$ from theorem \ref{th:et-ex}
%is injective.
%\end{corollary}
%\begin{proof}
%As noted in remark \ref{},
%for homotopy invariant presheave with GW-transfers $\cal F$,
%formula 
%$$\begin{array}{ccl}
%GWCor^{pair}&\longrightarrow& Ab\\
%(Y,U)&\mapsto& \Coker(\cal F(Y)\to \cal F(U))
%\end{array}$$
%defines homotopy invariant presheave on the category $GWCor^{pair}$,
%and since functor of injective limit is exact, then
%$$\Coker(\cal F(X_z)\to \cal F(X_z-Z_z))=\varinjlim\limits_{U\colon z\in U\subset X} \Coker(\cal F(U)\to \cal F(U-Z)).$$
%The claim follows.
%\end{proof}

\begin{proposition}\label{prop:et-ex:sur:GWCor} %label{prop:GWCor-etex}
Suppose 
$\pi$, $X$, $X^\prime$, $Z$, $Z^\prime$, $z$, $z^\prime$
%%a morphism $\pi$, schemes $X$, $X^\prime$, $Z$, $Z^\prime$ and points $z$, $z^\prime$,
%as in theorem \ref{th:et-ex}, then 
satisfy the assumptions as in theorem \ref{th:et-ex},
then
there is
a morphism $\Phi_r\in GWCor( (X_z,Z_z) , (X^\prime,Z^\prime) )$
such that $$\Phi \circ \pi_z \stackrel{\affl}{\sim} i_{z^\prime} \in GWCor( (\pri X_{\pri z},\pri Z_{\pri z}) , (\pri X,\pri Z) ) ,$$
where $i_z\colon (X_z,Z_z)\hookrightarrow (X,Z) $ denotes the canonical morphism of pairs,
and $\pi_z\colon (\pri X_{\pri z}, \pri Z_{\pri z}) \to (X_z,Z_z)$ is a morphism of Nisnevich neighbourhoods of local schemes induced by $\pi$.

%is considered as morphism $(\pri X_{\pri z}, \pri Z_{\pri z}) \to (X_z,Z_z)$.

\end{proposition}
\begin{proof}

We start by the same way as in proof of proposition \ref{prop:et-ex:inj:GWCor}; this means
that using lemma \ref{lm:et-ex:relCures}
we modify Nisnevich neighbourhood $(X^\prime,Z^\prime)\to (X,Z)$,
to a (relative) Nisnevich neighbourhood
$\varpi\colon (\tXp,{\tilde Z}^\prime)\to (\tilde X,\tilde Z)$
over some essential smooth local base $S$
quipped with good relative compactification
$\overline{\tilde{\pi} }\colon \ovtXp\to \ovtX$,
such that $\bXp$ and $\cal X$ are smooth over $S$
and there are trivialisations of relative canonical classes 
$\mu^\prime\colon \omega_S(\bXp)\simeq \cal O(\bXp)$ and $\mu\colon \omega_S(\cal X)\simeq \cal O(\cal X)$,
and 
such that there is a very ample bundle
$\cal O(1)$ on $\ovtX$,
such that 
$\ovvarpi^*(\cal O(1))$ is too very ample.
Similarly as in proof of the proposition \ref{prop:et-ex:inj:GWCor} to shortify notations 
let's denote $\ovvarpi^*(\cal O(1))$ by the same symbol $\cal O(1)$.

%In addition let's shrink $\pri U$ in such way that %$\pri Z =  \pri U\times_U Z$.

Denote $U=X_z$, $\pri U-\pri X_{\pri z}$, $\pi_z\colon \pri U\to U$
and let's shrink $\pri U$ in such way that $\pri Z_{\pri z}=\pi_z^{-1}(Z_z)$.
Then consider the base changes of $\tXp\to S$ along $U\to S$ and $U^\prime\to S$ 
(see digram \eqref{eq:et-ex:sur:basechange}),
and denote
$$(\bXp,\bZp)=(\tXp,\tZp)\times_S U, (\bXpp,\bZpp)=(\tXp,\tZp)\times_S \pri U.$$

Denote by $\pri\Delta\subset \bXp$ the graph of the canonical morphism $\pri U=\pri X_{\pri z}=\tXp_{\pri z}\to \tXp$.
Then $\bZp\cap \pri\Delta$ is equal to the diagonal $\Delta_{\pri Z_{\pri z}}\subset \pri Z_{\pri z}\times_S \pri Z_{\pri z}$, 
i.e. the closed subscheme in $\ovbXpp$ that is 
the graph of the composition $\pri Z_{\pri z}\to \pri U_{\pri z}=\tXp_{\pri z}\to\tXp$.

Next we want to apply lemma \ref{lm:et-ex:sur:FinFuncSup} to 
the morphism $\pi_z\colon (\pri U,\pri Z)\to (U,Z)$,
the curves $\bXp$ over $U$ and $\bXpp$ over $\pri U$,
the closed subschemes $\bZp\subset \bXp$, $\bZpp\subset\bXpp$
and the divisor $\pri\Delta\subset \bXpp$.
To do this we should find a regular function $\delta\in k[\cal Z]$, 
such that $Z(\delta)$ is equal to the subscheme $\Delta_{Z_z}$ that is 
the graph of the composition $Z_z\simeq \pri Z_{\pri z}\to \pri U_{\pri z}=\tXp_{\pri z}\to\tXp$
(i.e. to the diagonal in $Z_z\times_S Z_z$).
%Such function $\delta$ exists, s
Since $\bZp= \pri Z\times_S U\simeq Z\times_S U\hookrightarrow U\times_S U$, the scheme $\bZp$ can be identified with a closed subscheme in $U\times_S U$.
Let $\tilde{\delta}$ denotes the inverse image to $\bZp$ of some section of the line bundle $\mathcal L(\Delta_U)$ on $U\times_S U$ that zero divisor is diagonal.
Since the scheme $\bZp$ is finite over the local scheme $U$, and consequently any line bundle on $\bZp$ is trivial,
the section $\tilde{\delta}$ can be considered as a regular function that gives us the required function $\delta$.
%and so $\delta$ is inverse image of the section on $U\times_S U$ that divisor is diagonal.

Thus using lemma \ref{lm:et-ex:sur:FinFuncSup} we find a regular function 
$\pri f\in k[\bXp]$ 
that is relatively finite over $U$
and functions 
$f\in k[\bXpp\times\affl]$, 
$f_0,f_1\in k[\bXpp]$
that are relatively finite over $\pri U$,
such that 
\begin{equation}\label{eq:et-ex:comm-diag}
\pi_z^*(f^\prime)=f_0,\, i_0^*(f)=f_0,\, i_1^*(f)=f_1, 
\end{equation}
and such that
\begin{gather}
\pri Z \times_{\pri X} Z(f)\times_{\pri U} (\pri U -\pri Z_{\pri z})=\emptyset
\label{eq:et-ex:sur:Zfpair}\\
Z^\prime \times_{X^\prime} Z(f^\prime) \times_{\pri U} (\pri U-\pri Z_{\pri z})=\emptyset
\label{eq:et-ex:sur:Zppair}\\
Z(f_1)=\pri\Delta\amalg (Z(f_1)-\Delta),\, \pri Z\times_{\pri X} (Z(f_1)-\pri\Delta)=\emptyset
\label{eq:et-ex:sur:Z1}
\end{gather}
%\begin{gather}
%Z(f)\times_{\pri U\times\affl }(\pri U\times \affl-\pri Z_{\pri z}\times\affl)=
%(\pri X-\pri Z) \times_{\pri X} Z(f)\times_{\pri U\times\affl} (\pri U\times \affl-\pri Z_{\pri z}\times\affl)
%\label{eq:et-ex:sur:Zfpair}\\
%Z(\pri f)\times_{\pri U} (\pri U-\pri Z_{\pri z} )=
%(X^\prime-Z^\prime) \times_{X^\prime} Z(f^\prime) \times_{\pri U} (\pri U-\pri Z_{\pri z})
%\label{eq:et-ex:sur:Zppair}\\
%Z(f_1)=\pri\Delta\amalg (Z(f_1)-\Delta),\, (Z(f_1)-\Delta)=(\pri X-\pri Z)\times_{\pri X} (Z(f_1)-\pri\Delta)
%\label{eq:et-ex:sur:Z1}
%\end{gather}

Now similarly as in proposition \ref{prop:et-ex:inj:GWCor},
define the base changes %inverse images 
of the trivialisation $\mu$:
$\mu_{U}\colon \omega_{U}(\bXp)\simeq \cal O(\bXp)$, 
$\mu_{\pri U\times\affl}\colon \omega_{\pri U\times\affl}(\bXp\times\affl) \simeq \cal O(\bXp\times\affl).$
%So we get oriented curve $(\cal X,\mu_{X_z})$ over $X_z$.
Then we get the first one of three following diagrams, which is the diagram of oriented relative curves with finite functions (see def. \ref{def:OrCurFinFun}).
Next applying construction from proposition \ref{prop:constrOrCurFQCor}
we get quadratic spaces
\begin{multline*}
(k[Z(f^\prime)],q^\prime) = \langle \bXp, \mu_{U}, f^\prime  \rangle \in Q(\cal P(Z(f_0)\to U)),
(k[Z(f_0)],q_0) = \langle \bXpp, \mu_{\pri U}, f_0  \rangle \in Q(Z_1\to X_z),\\
(k[Z],q) = \langle \bXpp, \mu_{\pri U\times\affl}, f  \rangle \in Q(\cal P(Z(f)\to \pri U\times\affl)),\\
(k[Z(f_1)],q_1) = \langle \bXpp, \mu_{\pri U}, f_1  \rangle \in Q(\cal P(Z_1\to X_z)).
\end{multline*}
%Composing these quadratic spaces with injections and regular maps $v^\prime$ %regular map$v^\prime$, 
%and using lemma \ref{lm:QPairCor} and equalities \eqref{eq:et-ex:Zfpair}, and \eqref{eq:et-ex:lZpair} 
%we get GW-correspondences $\Phi$, $\Theta_0$, $\Theta$, $\widetilde\Theta_1$ between pairs:
%and due to commutativity of the diagram this GW-correspondences satisfy required relations:
\begin{equation*}
%\begin{gather}
%\mu\colon \cal O(X^\prime)\simeq \omega(X^\prime)\\
%\label{eq:et-ex:sur:diagrams}
\xymatrix{
\stackrel{(\mu,  f^\prime ) }{ \bX^\prime } \ar[d] & 
\stackrel{(\mu,  f_0) }{ \bX^{\prime\prime} } \ar[l] \ar[r]_{(0)_\bXpp} \ar[d] & 
\stackrel{(\mu, f) }{ \bXpp\times\affl } \ar[d] & 
\stackrel{(\mu, f_1) }{ \bXpp } \ar[l]^{(1)_\bXpp} \ar[d]
  \\
U & U^\prime \ar[l]\ar[r]^{(0)_U} & U^\prime\times\affl & U^\prime \ar[l]_{(1)_U}
\\
&& X^\prime & 
  \\
\stackrel{ (k[Z^\prime],q_{Z^\prime}) }{ Z^\prime } \ar[urr] \ar[d] & 
\stackrel{ (k[Z_0],q_{Z_0}) }{Z_0}\ar[ur] \ar[d] \ar[r] \ar[l] & 
\stackrel{ (k[Z],q_{Z}) }{Z}\ar[u] \ar[d] & 
\stackrel{ (k[Z_1],q_{Z_1}) }{Z_1}\ar[ul] \ar[d] \ar[l]
  \\
U & U^\prime \ar[l]^{\pi} \ar[r]_{(0)_{U^\prime}} & U^\prime \times\affl & U\ar[l]^{(1)_{{\pri X_{\pri z} }}}
\\
&& (X^\prime,Z^\prime) & 
  \\ \\
(X_z,X_z-Z_z) \ar[uurr]|{\Phi} & 
(X^\prime_{\pri z},X^\prime_{\pri z}-Z^\prime_{\pri z}) \ar[l]^{\pi} \ar[r]_{(0)_{U^\prime}} \ar[uur]|{\Theta_0} & 
(X^\prime_{\pri z},X^\prime_{\pri z}-Z^\prime_{\pri z}) \times\affl \ar[uu]|{\Theta} & 
(X_z,X_z-Z_z) \ar[l]^{(1)_{\pri U }} \ar[uul]|{\Theta_1}_{\sim \text{вложение } U^{\prime}\hookrightarrow X^\prime}
}
\end{equation*}%gather}
%This means that we define 
%\begin{multline*}
%(k[Z(f^\prime)],q^\prime) = \langle \bXp, \mu_{U}, f^\prime  \rangle \in Q(\cal P(Z(f_0)\to U)),
%%(k[Z(f_0)],q_0) = \langle \bXpp, \mu_{\pri U}, f_0  \rangle \in Q(Z_1\to X_z),\\
%(k[Z],q) = \langle \bXpp, \mu_{\pri U\times\affl}, f  \rangle \in Q(\cal P(Z(f)\to \pri U\times\affl)),\\
%(k[Z(f_1)],q_1) = \langle \bXpp, \mu_{\pri U}, f_1  \rangle \in Q(\cal P(Z_1\to X_z)),
%\end{multline*}
Next considering compositions with morphisms of the diagram of base changes
\begin{equation}\label{eq:et-ex:sur:basechange}
\xymatrix{
\pri X&\tXp\ar[d]\ar[l]^{\pri v} & \bXp\ar[l]^{\pri e}\ar[d] & \bXpp\ar[l]^{\ppri e}\ar[d] & \bXpp\ar[d]\ar[l]^{pr_{\affl}}\\
&S & U\ar[l] & \pri U\ar[l]^{\pi_z}& \pri U\times\affl \ar[l]
}\end{equation}
and the canonical injections 
$i_{Z(f^\prime)}\colon Z(f)\hookrightarrow \bXp$,
$i_{Z(f)}\colon Z(f)\hookrightarrow \bXpp\times\affl$
we get GW-correspondences,
\begin{gather*}
\Phi = [v^\prime \circ i_{Z(f^\prime)} \circ (k[Z(f^\prime)],q^\prime)] \in 
GWCor( ( U,U-{Z_{ z}} ) , (X^\prime,X^\prime-Z^\prime) ),\\
%\Phi_0 = [v \circ i_{Z}\circ (k[Z_0],q_0)],
\Theta = [v \circ i_{Z(f)}\circ (k[Z(f)],q)]\in 
GWCor( ( \pri U,\pri U-{\pri Z_{\pri z}} )\times\affl, (X^\prime,X^\prime-Z^\prime) ),
%\Phi_1 = [v \circ i_{Z}\circ (k[Z_1],q_1)],
\end{gather*}
which are GW-correspondences of pairs
by lemma \ref{lm:QPairCor} and equalities \eqref{eq:et-ex:sur:Zfpair} and \eqref{eq:et-ex:sur:Zppair}.
Then we get
\begin{multline*}
\Phi\circ \pi_z = [v^\prime \circ i_{Z(f^\prime)} \circ \langle \bXp, \mu_{U}, f^\prime  \rangle \circ \pi_z] = 
[v^\prime \circ \ppri e \circ i_{Z(f_0)} \circ \langle \bXpp, \mu_{\pri U}, f_0  \rangle] = \\ =
[v^\prime \circ \ppri e \circ i_{Z(f)} \circ \langle \bXpp\times\affl, \mu_{\pri U\times\affl}, f  \rangle \circ i_0] = 
%[v \circ i_{Z}\circ (k[Z],q)\circ i_0] = 
\Theta\circ i_0,
\end{multline*}
\begin{multline*}
\Theta\circ i_1=
[v^\prime \circ \ppri e \circ i_{Z(f)} \circ \langle \bXpp\times\affl, \mu_{\pri U\times\affl}, f  \rangle \circ i_1] = 
[v^\prime \circ \ppri e \circ i_{Z(f_1)} \circ \langle \bXpp, \mu_{\pri U}, f_1  \rangle] = \\ =
[v^\prime \circ i_{\pri\Delta}\circ (k[\Delta^\prime],u^\prime)] + 
[v^\prime \circ i_{Z(f_1)-\pri\Delta}\circ (k[Z(f_1)-\pri\Delta],q_1\big|_{Z(f_1)-\pri\Delta})] = \\ = 
[v \circ i_{\Delta}\circ (k[\Delta],\pri u)]=
[i_{X_z}\circ \langle \pri u\rangle]\in GWCor( (\pri U,\pri U-\pri Z_{\pri z}), (\pri X,\pri X-\pri Z) )
\end{multline*}
where the second and the third equalities follows from lemma \ref{lm:QPairCor} and \eqref{eq:et-ex:sur:Z1}.
Then we replace $\Phi$ and $\Theta$ by $\Phi\circ \langle u^{-1} \rangle$
and $\Theta\circ \langle u^{-1} \rangle$,
where $u\in k[U]^*$ is any invertible regular function such that $u\big|_{Z_z}=u^\prime\big|_{Z^\prime_{\pri z}}$ (we use that $Z^\prime\simeq Z$).

To finish the proof of the proposition it is enough to prove that 
$$[j\circ \langle \pri u/u\rangle ]\stackrel{\affl}{\sim} [j] \in WCor( (\pri X,\pri X-\pri Z) , (V,V-Z) ),$$
where $j\colon (\pri U,\pri U-\pri Z_{\pri z})\to (\pri X,\pri X-\pri Z)$.
%Let's prove that it
%$[i_{(\pri U,\pri U-\pri Z_{\pri z})\to (\pri X,\pri X-\pri Z)}\circ \langle \pri u/u\rangle ] \in WCor( (\pri X,\pri X-\pri Z) , (V,V-Z) )$.
%%Let's prove that it
%is equal to $i_{(\pri U,\pri U-\pri Z_{\pri z})\to (V,V-Z)}$ up to $\affl$-homotopy.
%$[\langle \tilde u\rangle \circ i_{(\pri U,\pri U-\pri Z_{\pri z})\to (V,V-Z)}]=id_{(V,V-Z)}$.
Consider affine Zariski neighbourhood of $\pri z$ in $\pri X$ with a lift of the function $u$ to a regular invertible function $\tilde u$, and 
consider two-degree covering $c\colon W=Spec\,k[V][w]/(w^2-\tilde u)\to V$, which is etale in some neighbourhood of $Z$, since $\pri u/u(\pri z)=1$ and $char\,k\neq 2$.
Let's shrink $V$ and $W$ in such way that $c$ becomes etale,
and denote by $\ppri Z\subset\ppri U$ the closed subscheme defined by the ideal $(w-1)$ in $k[\pri Z][w]/(w^2-1)$. 
Then $c\colon (\ppri U,\ppri z) \to (\pri U,\pri z)$ is Nisnevich neighbourhood.

Since the inverse image of $\tilde u$ in $k[W]$ is equal to the square function $w^2$, 
it follows that 
$[\langle \tilde u\rangle\circ id_{(V,V-Z)}\circ d]=
[\langle w^2\rangle\circ id_{(V,V-Z)} ]=
[\langle 1\rangle\circ id_{(V,V-Z)}\circ dc]\in GWCor( (W,W-Z) , (V,V-Z) )$.
By proposition \ref{prop:et-ex:inj:GWCor} there is a GW-correspondence $\Phi\in GCor( (\pri U,\pri U-{\pri Z_{\pri z}}) , (W,W-Z) ) $
left inverse to the class of morphism $c$.
Hence if we denote 
$j^{\pri X}_V\colon (V,V-Z)\to (\pri X,\pri X-\pri Z)$, $j^V_{\pri U}\colon (\pri U,\pri U-\pri Z_{\pri z})\to (V,V-Z)$,
\begin{multline*}
[j  \circ \langle u\rangle\circ id_{(\pri U,\pri U-\pri Z_{\pri z})}]=
[j^{\pri X}_V \circ \langle \tilde u\rangle\circ j^V_{\pri U}]=\\=
[j^{\pri X}_V \circ \langle \tilde u\rangle\circ id_{(V,V-Z)}\circ c\circ \Phi]=
[j^{\pri X}_V \circ \langle 1\rangle\circ id_{(V,V-Z)}\circ c\circ \Phi]=\\=
[j^{\pri X}_V \circ j^V_{\pri U}]= [j]\in WCor( (\pri U,\pri U-{\pri Z_{\pri z}}) , (\pri X,\pri X-\pri Z) )
\end{multline*}

\end{proof}

\begin{proof}[Proof of the theorem \ref{th:et-ex}]
As noted in remark \ref{rm:GWCorFpair},
for homotopy invariant presheave with GW-transfers $\cal F$,
the formula 
$$\begin{array}{ccl}
GWCor^{pair}&\longrightarrow& Ab\\
(Y,U)&\mapsto& \Coker(\cal F(Y)\to \cal F(U))
\end{array}$$
defines the homotopy invariant presheave on the category $GWCor^{pair}$,
and since the injective limit functor is exact,
$$\Coker(\cal F(X_z)\to \cal F(X_z-Z_z))=\varinjlim\limits_{U\colon z\in U\subset X} \Coker(\cal F(U)\to \cal F(U-Z)).$$

Hence the injectivity of the homomorphism $\pi^*$ follows from proposition \ref{prop:et-ex:inj:GWCor}, and
the surjectivity from proposition \ref{prop:et-ex:sur:GWCor}.
\end{proof}

\section{Injectivity}\label{sect:Injth}

\begin{theorem}\label{th:affZar-inj}
Let 
  $\cal F$ is be a homotopy invariant presheave with GW-transfers over field $k$
  and $K$ be a geometric extension $K/k$ (i.e. field of functions of some variety).
Then 
  for any Zariski open subschemes $U\subset V\subset\affl_K$
  the restriction homomorphism   
    $$i^*\colon {\cal F(V)}\to {\cal F(U)}$$
  is injective, 
  where $i\colon U\hookrightarrow V$ denotes the open immersion.  
\end{theorem}

\begin{proposition}\label{prop:affZar-inj:GWCor}
For 
a morphism $i\colon U\hookrightarrow V$ 
satisfying the assumptions of theorem \ref{th:affZar-inj}
there is
a morphism $\Phi\in GWCor( V , U )$
such that $$i \circ \Phi \stackrel{\affl}{\sim} id_{V} \in GWCor( V , V ).$$

\end{proposition}\begin{proof}
Lemma \ref{lm:GWCor:locbase.} implies that is is enough to consider the case $K=k$. %we can assume $K=k$.
%
%Consider immersion $\affl\subset \prl$ and 
Denote divisors on relative projective line
$
\infty_V = \prl_V\setminus \affl_V,\;
T = \affl_{V} \setminus V\times V\subset \prl_V,\;
D = V\times V \setminus U\times V \subset \prl_V,\;
\Delta = \Gamma(V\hookrightarrow \prl)
$,
i.e. $\Delta$ is diagonal in $V\times V$,
and let's fix sections 
$$
\mu,\nu,\delta\in\Gamma(\prl_{V},\cal L(1))\colon
div_0\,\mu = \infty_V ,\; 
div_0\,\nu = 0_V,\;
div_0\,\delta = \Delta,\;
\nu\big|_{\infty_V}=\delta\big|_{\infty_V}
.$$ 

By lemma \ref{lm:corSerreth}
for a sufficiently large $l$ there are sections
$$\begin{aligned}
s_0\in \Gamma( \prl_V, \cal O(l) ),\;&
s_0\big|_{\infty_V} = \nu^{l},\;
s_0\big|_{T} = \delta \mu^{l-1},\;
s_0\big|_{D} = \mu^{l},\;
\\
g\in \Gamma( \prl_V, \cal O(l-1) ),\;&
g\big|_{\infty_V} = \nu^{l-1},\;
g\big|_{T\amalg D} = \mu^{l-1},\;
g\big|_{\Delta} = \mu^{l-1}
.\end{aligned}$$

Next define $s = s_0(1-t) + \delta g t\in \Gamma(\prl_V\times\affl,\cal O(l) )$,
then 
$
s\big|_{\infty_V\times\affl} = \nu^l,\; s\big|_{T\times\affl} = \delta\mu^{l-1}
$.

By lemma \ref{lm:sect-finfunct} the functions 
$s_0/\mu^l\in k[\affl_{V}]$ and $s/\mu^l\in k[\affl_{V\times\affl}]$
are relatively finite. 
So we can apply the construction from proposition \ref{prop:constrOrCurFQCor} and put
\begin{multline*}
Q_0=\langle dy,s_0/\mu^l\rangle=(k[Z_0],q_0)\in Q(\cal P(Z_0\to V)),\\
Q=\langle dy,s/\mu^l\rangle=(k[Z],q)\in Q(\cal P(Z\to V\times\affl))\\
Q_1=\langle dy,\delta g/\mu^l\rangle=(k[Z_1],q)\in Q(\cal P(Z_1\to V))
,\end{multline*}
where $$Z_0 = Z(s_0)\subset \affl_{V},\quad Z = Z(s)\subset \affl_{V\times\affl}, Z_1 = Z(\delta g)\subset \affl_{V},$$
and $dy$ denotes the differential on the relative affine line defined by the coordinate function.

Since 
$V\cap T=\infty$, it follows that  
$s\big|_{T\times\affl}=\delta\mu^{l-1}\big|_{T\times\affl}$ is invertible,
and
$s_0\big|_{T\amalg D}=\delta\mu^{l-1}\big|_{T}\amalg \mu^{l}\big|_{D}$ is invertible.
Hence 
we get injections
$$
i_{Z_0}\colon Z_0\to U\times V,\;
i_{Z}\colon Z\hookrightarrow V\times V\times\affl
.$$
Since
$g\big|_{\Delta}=\mu^{l-1}$, it follows that 
$Z(\delta g)=\Delta\amalg Z(g)$,
and consequently there is a splitting of the quadratic space 
$$Q_1=Q_{\Delta}\oplus Q_{Z(g)},\; Q_{\Delta} = (k[\Delta],u),\, u\in k[\Delta]^*,\; Q_{Z(g)} = (k[Z(g)],q_{Z(g)}).$$
Since $g\big|_{D}=\mu^{l-1}$, it follows that $Z(g)\subset U\times V$
and
there is an injection $i_{Z(g)}\colon Z(g)\hookrightarrow U\times V$.

Hence
quadratic spaces 
$i_{Z_0} \circ Q_0$, $i_{Z(g)}\circ Q_{Z(g)}$ and $i_Z\circ Q$
define GW-correspondences
\begin{gather*}
\widetilde{\Phi}=[i_{Z_0} \circ Q_0]-[i_{Z(g)} \circ Q_{Z(g)}]\in GWCor( V, U ),\\
\widetilde{\Theta}=[i_Z\circ Q] - [i\circ i_{Z(g)}\circ Q_{Z(g)}\circ pr] \in GWCor( V\times\affl, V )
,\end{gather*}
where $pr\colon V\times \affl\to \affl$.

Then
\begin{multline*}
i \circ \widetilde{\Phi}=
[i\circ i_{Z_0}\circ \langle dy,s_0/\mu^l \rangle] - [i\circ i_{Z(g)} \circ Q_{Z(g)}]=
[i\circ i_{Z_0}\circ i_0^*(\langle dy,s/\mu^l \rangle)] - \\ [i\circ i_{Z(g)} \circ Q_{Z(g)}]=
[i_{Z}\circ \langle dy, s/\mu^l \rangle\circ i_0] - [i\circ i_{Z(g)} \circ Q_{Z(g)}\circ pr\circ i_0]=
\widetilde{\Theta}\circ i_0
.\end{multline*}
On other side 
$$
\widetilde{\Theta}\circ i_1=
[i_{Z_1}\circ Q_1]  - [i\circ i_{Z(g)} \circ Q_{Z(g)}]=
[(k[\Delta],u)]
,$$
Thus if we put 
$$
\Phi=\widetilde{\Phi}\circ [\langle u^{-1}\rangle],\;
\Theta=\widetilde{\Theta}\circ [\langle u^{-1}\rangle],\;
,$$
then 
$\Theta\circ i_0=i \circ \Phi, \Theta\circ i_1 = [(k[\Delta],1)]=id_{V},$
and so $i \circ \Phi\stackrel{\affl}{\sim} =id_{V}$. % by homotopy $\Theta$.

\end{proof}

\begin{theorem}\label{th:locZar-inj}
For any essential smooth local scheme $U$,
and a closed subscheme $Z\subset U$,
the restriction homomorphism
    $$i^*\colon {\cal F(V)}\to {\cal F(U)}$$
is injective.
\end{theorem}

\begin{proposition}\label{prop:locZar-inj:GWCor}
For any smooth variety $X$, a point $z\in X$, and a closed subscheme $Z\subset X$,
there is a
morphism $\Phi\in GWCor( X_x , X-Z )$,
such that $j \circ \Phi \stackrel{\affl}{\sim} i_z \in GWCor( (X,Z)_z , (X,Z) ) $,
where $i_z\colon X_z\hookrightarrow X $ 
and $j\colon X-Z\hookrightarrow X$.

\end{proposition}\begin{proof}
In the same way as in lemma \ref{lm:et-ex:relCures}, %!!short
we change $X$ to a relative smooth curve with good compactification. 
Let's repeat this construction in this situation:
Firstly, shrink $X$ in such way that canonical classe of $X$ becomes trivial.
Next, consider decomposition
$X\xrightarrow{u} \overline{X}\xrightarrow{p} \aff^d,$
of the etale morphism $e\colon X\to \aff^d$ ($d=dim\,X$),
where $u$ is dense open immersion and $p$ is finite, given by proposition \ref{prop:ZarMianTh}. %??ref Serre
Then using lemma \ref{lm:relCur:comp:projection}
find projection $pr\colon \aff^d\to \aff^{d-1}$
such that the restrictions $p(\overline X\setminus X)\to \aff^{d-1}$ and $p(Z)\to \aff^{d-1}$ are finite.
%So we get etale morphism $e\colon X\to \aff^d$ and projection $pr\colon \aff^d\to \aff^{d-1}$,
%such that $e$ and $e\circ \pi$ are finite over open subscheme $\aff^d\setminus P$, and $P$ is finite over $\aff^{d-1}$. 

Now using the base change
along the projection %of local scheme 
$U=X_z\to \aff^{d-1}$
%$S=\aff^d_s$, where $s=pr(e(z))$,
we get the curve $\ovbX=U\times_{\aff^{d-1}} \overline X$ with the finite morphism
$\pi\colon \ovbX\to \affl_U$
and the smooth open subscheme $\cal X=U\times_{\aff^{d-1}} X$
such that $p$ is etale on $\cal X$ 
and $\ovbX\setminus \cal X$ is finite over $U$.
%$e_U\colon \cal X\to \affl_U$ that
%is finite over open subscheme $\affl_U \setminus P$, where $P= U\times_{\aff^{d-1}} p(\overline X\setminus X)$ is finite over $U$.

Next we consider immersion $\affl_U\hookrightarrow \prl_U$
and again applying Zariski main theorem (proposition \ref{prop:ZarMianTh}) 
%Zariski main theorem %8.12.6  В EGAiv  Corollary 18.12.12. ("Main theorem" de Zariski)
we get decomposition 
$\ovbX\xrightarrow{j} \overline{\ovbX}\xrightarrow{\overline{\pi}} \prl_U$
$$\xymatrix{
\affl_U\ar@{^(->}[d] &\cal X \ar[l]^{\pi}\ar@{^(->}[d]^{j}\ar[r]^{ev}& X\\
\prl_U &\overline{\cal X}\ar[l]^{\overline{\pi}} & 
}$$

Let $\Delta_z = \Gamma(z\to X)\in \cal X$ be a closed point (the diagonal in $z\times_{\aff^d} z$). 
Since $\overline{\pi}$ is finite, then $\cal O(1)$ is an ample bundle on $\overline{\ovbX}$.
Using lemma \ref{lm:corSerreth} and replacing $\cal O(1)$ by a some power we can assume 
that there is 
$d\in \Gamma(\prl_U,\cal O(1))\colon Z(d)\supset \ovbX\setminus \cal X, Z(d)\not\ni z$,
and let's redenote $\cal X=\overline{\ovbX}-Z(d)$. And let's redenote $\overline{\ovbX}$ by $\ovbX$.

Now consider the closed subschemes $\Delta, \cal Z\subset \ovbX$:
$$
U\simeq \Delta = \Gamma(U\to X)  ,\;
\cal Z = { U\times_{\aff^{d-1}} \overline Z }
.$$
($\cal Z$ is a closed subscheme in $\ovbX$, since $\overline Z$ is finite over $U$)

By lemma \ref{lm:corSerreth} for some sufficiently large $l$, there are sections
$$\begin{aligned}
g\in \Gamma(\ovbX, \cal L(\Delta)^{-1}\otimes \cal O(l))\colon\;&
Z(g)\cap (D\cup \Delta\cup \cal Z) = \emptyset,\\
s_0\in \Gamma(\ovbX, \cal \cal O(l))\colon\;&
s_0\big|_{D} = \delta g,\;
Z(s_0)\cap {\cal Z}=\emptyset
.\end{aligned}$$
Define $s = s_0(1-t) + \delta g t \in \Gamma(\ovbX\times\affl,\cal O(l))$,
\begin{gather*}
i_{Z_0}\colon Z_0=Z(s_0)\to \cal X - \cal Z,\; 
i_{Z(g)}\colon Z(g)\to \cal X - \cal Z,\; 
i_{\Delta}\colon \Delta\to \cal X,\\
Q_0=\langle s_0/d^l \rangle \in Q(\cal P(Z_0\to U)), 
Q=\langle s/d^l \rangle \in Q(\cal P(Z\to U)), 
Q_1=\langle \delta g/d^l \rangle \in Q(\cal P(Z(\delta g)\to U)), 
\end{gather*}
Since $Z(\delta g)=\Delta \amalg Z(g)$, it follows that
$$Q_1 = (k[\Delta],u)\oplus Q_{Z(g)},\; Q_{Z(g)} = (k[Z(g)],q_{Z(g)}).$$

Define
\begin{multline*}
\tilde{\Phi} = [ev\circ i_{Z_0}\circ Q_0] - [ev\circ i_{Z(g)}\circ Q_{Z(g)}] \in GWCor(U,X-Z),\\
\tilde{\Tilde} = [ev\circ i_{Z}\circ Q] - [ev\circ i_{Z(g)}\circ Q_{Z(g)}\circ pr]\in GWCor(U,X),
\end{multline*}
where $pr\colon U\times \affl\to U$.

Then
$$\tilde{\Theta}\circ i_0 = \tilde{\Phi}, \tilde{\Theta}\circ i_1 = [(k[\Delta], u)] = [i\circ \langle u\rangle],$$
where $i_0,i_1\colon U\to U\times \affl$ denotes zero and unit sections.

Hence the morphisms 
$\Phi=\tilde\Phi\circ [\langle u^{-1}\rangle]$, 
$\Tilde=\tilde\Tilde\circ [\langle u^{-1}\rangle]$
give the required 'right inverse' GW-correspondence up to a homotopy.
\end{proof}

%\section{Sheafification on affine line} 

\section{Strictly homotopy invariance of associated sheave}\label{sect:StrHomInv}

%Here we give formulations of the main results of the article that are consequences of the excision and injectivity theorems from the previous sections.
%The deduction is standard now, so to shortify the text we give only brief comments instead of the full proof, and refer the reader to \cite[sect. 4.5, 5.1, 5.2]{VSF_CTMht_Ctpretr} for the original version, to \cite{GP15} where the same results are proved for presheaves with framed transfers, and to the our preprint \cite{AD_StrHomInv} for the full proofs in the considered case.

\begin{theorem}\label{th:affcoh}
Let $\cal F$ be a homotopy invatiant presheave with GW-transfers,
then
$$\widetilde{\cal F}_{Nis}\big|_{\affl_K}\simeq \cal F\big|_{\affl_K},\;
h^i_{Nis}(\widetilde{\cal F}_{Nis})\big|_{\affl_K}\simeq 0, \forall i>0
,$$
for any geometric extension $K/k$.
\end{theorem}
\begin{proof}

Let's denote $X=\affl_K$.
Consider the exact sequence of sheaves
\begin{equation}\label{eq:afflresolv}
0\to \widetilde{\cal F}_{Nis}\rightarrow 
\eta_*(\cal F(\eta)\stackrel{d}{\rightarrow}
\bigoplus\limits_{z\in \affl_K} 
z_*\left(\frac{\cal F(X^h_z-z)}{\cal F(X^h_z)}\right)\to 0, 
\end{equation}
where $z$ ranges over the set of closed points of $\affl_K$,
The sequence \eqref{eq:afflresolv} gives a flasque resolvent of $\widetilde{\cal F}_{Nis}$.
(Note that $\cal F(X^h_z-z)=\cal F(eta)$ since $dim\, X=1$.)

Let's compute cohomology presheaves of $\widetilde{\cal F}_{Nis}$ on $\affl_K$ using this resolvent.
Let $U\subset \affl_K$ be an open subscheme,
then 
$$
\widetilde{\cal F}_{Nis}=H^0(U,\widetilde{\cal F}_{Nis})=\Ker(
\cal F(\eta)\to \bigoplus\limits_{z\in U} 
z_*\left(\frac{\cal F(X^h_z-z)}{\cal F(X^h_z)}\right) )
.$$
So $H^0(U,\widetilde{\cal F}_{Nis})$ is the subset in $\cal F(\eta)$ consisting of 
elements $a$ that has a lift to a germ at each closed point of $U$.

Let $a$ be such an element in $\cal F(\eta)$.
Then for some open $V\subset U$ there is a section $\pri a\in \cal F(V)$, $\eta^*(\pri a)=a$, 
and using injectivity in the Zariski and etale excision on $\affl_K$ (corollary \ref{cor:afflzar-ex}, theorems \ref{th:et-ex}) we can consequently attach points of the complement $U\setminus V$ to $U$ and find an element $\tilde a\in \cal F(U)\colon \eta^*(\tilde a)=a$.
Thus homomorphism $\cal F(U)\to \widetilde{\cal F}_{Nis}$ is surjective.

On other hand, the injectivity on affine line (theorem \ref{th:affZar-inj}) implies that the composition 
$\cal F(U) \to \widetilde{\cal F}_{Nis}(U)\to \cal F(\eta)$ is injection.
Thus $\cal F(U)\to \widetilde{\cal F}_{Nis}$ is isomorphism.

Since the length of the resolvent \eqref{eq:afflresolv} is 2, it follows that 
$H^i_{Nis}(U,\widetilde{\cal F}_{Nis})=0$ for $i>1$.
Now let's prove that
$$
H^1(U,\widetilde{\cal F}_{Nis})=
\Coker(
\cal F(\eta)\to \bigoplus\limits_{z\in U} 
z_*\left(\frac{\cal F(X^h_z-z)}{\cal F(X^h_z)}\right) )=0
.$$
%That fact that it is zero is equivalent to that 
To do this it is enough to show that for any finite set of elements $a_i\in \cal F(X^h_{z_i}-{z_i})/\cal F(X^h_{z_i})$, $i=1\dots n$ 
there is an element $b\in \cal F(U-\{z_1,\dots z_n\})$, such that
$z^*(b) = a_i$, where 
$$z^*\colon \cal F(U-\{z_1,\dots z_n\})/\cal F(U-\{z_1,\dots \hat{ z_i}\dots z_n\})
\to \cal F(X^h_{z_i}-{z_i})/\cal F(X^h_{z_i})
. $$
The claim follow from the surjectivity of the excision homomorphisms in corollary \ref{cor:afflzar-ex} and theorem \ref{th:et-ex}.
\end{proof}

\begin{theorem}\label{th:HIAssZar}
The Nisnevich sheaf $\widetilde{\cal F}_{Nis}$ associated with homotopy invariant presheave with GW-transfers is homotopy invariant.
\end{theorem}
\begin{remark}For Witt-correspondences this was proved in \cite{AD_WtrSh}.\end{remark}
\begin{proof}
Let $X\in Sm_k$ and $K=k(X)$.
The theorem states that for any $X\in Sm_k$, $i_X^*\colon \widetilde{\cal F}_{Nis}(\affl_X)\to \widetilde{\cal F}_{Nis}(X)$ is isomorphism (for zero section $i_X\colon X\to \affl_X$).

Since the projection $\affl_X\to X$ is right inverse for $i_X$, the homomorphism $i_X^*$ is surjective.
 
To prove injectivity consider the commutative square
%%&^Diagram
$$\xymatrix{
\widetilde{\cal F}_{Nis}(\affl_X) \ar@{^(->}[r]^{J^*}\ar[d]^{i^*_{X}}&\widetilde{\cal F}_{Nis}(\affl_K)\ar[d]^{i^*_{K}}\\
\widetilde{\cal F}_{Nis}(X)\ar@{^(->}[r]^{j^*}&\widetilde{\cal F}_{Nis}(K)\\
}
,$$
where $K=k(X)$.
Theorem \ref{th:locZar-inj} yields that the horizontal arrows are injections,
and the right vertical arrow is isomorphism by theorem \ref{th:affcoh}.
The claim follows.
\end{proof}

Now we prove the main result of the article:
\begin{theorem}\label{th:strhominv}
%Nisnevich sheafication
%  of any homotopy invariant presheave with GW-transfers
%is strictly homotopy invariant.
For any homotopy invariant presheave with GW-transfers $\cal F$,
the presheaves of Nisnevich cohomologies of associated sheaf $h^i_{nis}( \widetilde{\cal F}_{nis})$ 
are homotopy invariant for all $i\geq 0$,
i.e. $$H_{nis}^*( \affl\times X , \widetilde{\cal F}_{nis} ) = H_{nis}^*( X , \widetilde{\cal F}_{nis} )$$ for $X\in Sm_k$.
\end{theorem}
The deduction of this theorem from excision theorems \ref{th:afflUzar-ex} and \ref{th:et-ex} is similar 
to the original proof in the case of $Cor$-correspondences. 
%We prove the claim consequently 
%We deduce it from the following two particular cases 
%of point and local schemes, 
%i.e. $X=pt_K$ for geometric extensions $K/k$ 
%and $X$ is essential smooth local scheme.
We start with the case of $X=Spec\,K$ for a geometric extension $K/k$:
\begin{lemma}\label{lm:strinv-pointbase}
For a homotopy invariant presheave with GW-transfers $\cal F$ and any geometric extension $K/k$
$$
\widetilde{\cal F}_{nis}\big|_{\affl_K} \simeq \cal F\big|_{\affl_K},\;
H^i_{nis}( \affl_K , \widetilde{\cal F}_{nis}) \simeq 0 ,\quad i>0
.$$
\end{lemma}
\begin{proof}
This is particular case of the theorem \ref{th:affcoh}.
\end{proof}
%This is starting point for our proof. % that is presented at the end of the section.
The next step is the case of essential smooth local scheme $X$:
%Firstly we prove some 
\begin{lemma}\label{lm:strinv-locbase}
For homotopy invariant presheave with GW-transfers $\cal F$ and essential smooth local henselian scheme $X$
$$
\widetilde{\cal F}_{nis}(\affl_X) \simeq \cal F(\affl_X),\;
H^i_{nis}(\affl_X, \widetilde{\cal F}_{nis}) \simeq 0 ,\quad i>0
.$$
\end{lemma}
\begin{remark}
%1)
%For a explanation of the same reasoning that doesn't use local schemes and deals only with varieties, look at \cite{} section 5 .
%2)
%In the proof of theorem \ref{th:strhominv} we use only computation of global sections i.e. cohomology groups $H^i_{nis}$.
% but this lemma also is usefull for cancellation theorem
\end{remark}

To prove of the last lemma we use the following notations: % presheaves: 
\begin{definition}\label{def:F-1F(affl/0)}
For any presheave (with GW-transfers) let's denote %we define tow presheaves (with GW-transfres) as follows
\begin{gather*}
\cal F_{-1}(-) = \Coker(\cal F(-) \to \cal F(-\times \bGm) )\\
\cal F(-\times \affl/0) = \Coker(\cal F(-) \to \cal F(-\times \affl) )
\end{gather*}
This defines the presheaves with GW-transfers.
(Note that we consider this just as notation,
though these presehave are in fact internal-hom-presheaves represented by $\bZ(\bGm)/\bZ(1)$ and $\bZ(\affl)/\bZ(0)$.) 
%?\bZ{GWtr}
\end{definition}
\begin{lemma}\label{lm:F1F(affl/0)}
The functors 
$\cal F\mapsto \cal F_{-1}$ and $\cal F\mapsto \cal F(-\times\affl/0)$
are exact in the category of presheaves (with GW-transfers).

If $\cal F$ is homotopy invariant %presheaves with GW-transfres 
or if it is a Nisnevich sheaf, then presehaves
$\cal F_{-1}$ and $\cal F(-\times \affl/0)$ are of such type too.
\end{lemma}
\begin{proof}
The claim follows immediate from that the canonical projections $\bGm\to pt$ and $\affl\to pt$ have splitting by unit sections,
and so the homomorphisms $\cal F(-)\to \cal F(-\times\bGm)$ and $\cal F(-)\to \cal F(-\times\affl)$ are surjective and splitting.
\end{proof}
\begin{lemma}\label{lm:strhominv:restrinj}
Let $U$ be an open subscheme of an essentially smooth local henselian scheme $X$,
such that $Z = X\setminus U$ is essentially smooth and $H^{i-1}(Z\times \affl/0 , \cal F)=0$ for some $i>0$.
%Let $U$ be open subscheme of smooth variety (or essentially smooth scheme) $X$,
%such that $Z = X\setminus U$ is smooth (or essentially smooth) and $H^{i-1}(Z\times \affl/0 , \cal F)=0$ for some $i>0$.
Then the restriction homomorphism
$$H^i_{nis}(X \times \affl/0 ,\widetilde{\cal F}_{nis} )\to H^i_{nis}( U \times \affl/0, \widetilde{\cal F}_{nis} )$$
is injective.
\end{lemma}
\begin{proof}[Proof of lemma \ref{lm:strhominv:restrinj}.]
%Let's for any smooth (essentially smooth) scheme $Y$ denote by $Y_{BigNis}$ the big Nisnevuch(?) site over $Y$,
%and denote by $\cal E\big|_Y$ restriction of the any presheaf $\cal E$ to this site.
Denote by %$i_Y\colon (Y\times Z)_{nis}\hookrightarrow (Y\times X)_{nis}$,  $j\colon (Y\times U)_{nis}\hookrightarrow (Y\times X)_{nis}$, 
$i_Y\colon (Z\times Y)_{nis}\hookrightarrow (X\times Y)_{nis}$,  $j\colon (U\times Y)_{nis}\hookrightarrow (X\times Y)_{nis}$, 
the morphisms of small Nisnevich sites, for any $Y\in Sm_k$.
 
%The Zariski excision on the relative affine line over the local base (theorems \ref{th:et-ex}) and the etale excision (\ref{th:afflUzar-ex})
%yields the natural isomorphism of sheaves % (for $Y$.
%$$
%{\Coker_{ShNis}(\cal F\xrightarrow{\varepsilon} j_*(j^*( \cal F\big|_{X\times Y} )) )}_{nis}
%\simeq 
%i_*( \cal F_{-1}\big|_{Z\times Y} )$$
%(see \cite[lemma 8.11]{AD_StrHomInv} for details).

The Zariski excision on the relative affine line over the local base (theorems \ref{th:et-ex}) and the etale excision (theorem \ref{th:afflUzar-ex})
yields the following:
\begin{sublemma}
There is the following natural isomorphism of sheaves
$$
{\Coker_{ShNis}(\cal F\xrightarrow{\varepsilon} j_*(j^*( \cal F\big|_{X\times Y} )) )}_{nis}
\simeq 
i_*( \cal F_{-1}\big|_{Z\times Y} )
$$
for any $Y\in Sm_k$.
\end{sublemma}
\begin{proof}
Since $X$ is an essentially smooth local henselian scheme, then there is an isomorphism
$(f,p)\colon X \simeq (Z\times\affl)^h_Z$
where $Z$ is considered as the subscheme of $Z\times\affl$ via the zero section %embedding by zero section

Let $V\in X_{nis}$. Denote the canonical morphism by $v\colon V\to X$, and let $Z^\prime = V\times^{(f,p)\circ v,i}_{Z\times\affl} Z$. % (i.e. $Z^\prime$ is the  preimage of $Z$ in $V$). 
Next consider the Nisnevich neighbourhood $V^\prime$ of $Z^\prime$ in $V$ defined as
$V^\prime = Z^\prime\times_Z V - \Delta_Z^\prime$, where $\Delta_{Z^\prime}$ denotes the diagonal in $Z^\prime\times Z^\prime$.
The sequence of etale morphisms 
\begin{equation}\label{eq:lm:str:NeighSeq}
 (V,Z^\prime)\leftarrow (V^\prime, Z^\prime)  \xrightarrow{ p\times id_{Z^\prime} } (Z^\prime\times\affl,Z^\prime) 
\end{equation}
induce homomorphisms
\begin{equation}\label{eq:lm:str:Fneighseq}
\cal F(V-Z^\prime)/\cal F(V) \rightarrow
\cal F(V^\prime-Z^\prime)/\cal F(V^\prime) \leftarrow
\cal F( (Z^\prime\times\affl)_{x}-Z^\prime )/\cal F((Z^\prime\times\affl) )
.\end{equation}
Since \eqref{eq:lm:str:NeighSeq} is natural in $V$, 
this defines homomorphisms of presheaves
$$\Coker_{PreSh}( \cal F\xrightarrow{\varepsilon} j_*(j^*(\cal F\big|_X)) )_{nis} ) 
\leftarrow 
\cal E 
\rightarrow 
i_*( \cal F_{-1}\big|_{Z} ).$$

Now consider the groups of germs at the point $x\in V$, i.e. substitute the henselisation $V^h_x$ in the equality \eqref{eq:lm:str:NeighSeq}.
If $x\not \in Z^\prime$, then all three germs are trivial.
Otherwise 
the first homomorphism becomes isomorphism by the definition of the Nisnevich neighbourhood $V^\prime$;
hence
consequently applying theorems \ref{th:et-ex} and \ref{th:afflUzar-ex} we get that the second homomrphism in \eqref{eq:lm:str:Fneighseq} is isomorphism too:
$$\cal F(V^h_x-Z^\prime_x)/\cal F(V^h_x)\simeq 
\cal F( (Z^\prime\times\affl)_{x}-Z^\prime )/\cal F((Z^\prime\times\affl)_{x}) \simeq 
\cal F( (Z^\prime\times\affl)_{x}-Z^\prime )/\cal F((Z^\prime\times\affl) ).$$

\end{proof}

Thus for any $Y\in Sm_k$ we get the short exact sequence of Nisnevich sheaves 
$
\cal F\big|_{X\times Y}\xrightarrow{\varepsilon} j_*(j^*(\cal F\big|_{X\times Y}))\to i_*(\cal F_{-1}\big|_{Z\times Y})
$. Hence there is the same sequence for any $Y$ in the idempotent completion of $Sm_k$.
%\begin{equation}%\label{ZXUseq}
%\end{equation}
Substituting $Y= \affl/0$ we get the long exact sequence of Nisnevich cohomology groups:
%$$
%\cdots \to H^{i-1}_{Nis}( X\times Y ,i_*(\cal F_{-1}) )  \to 
%H^i_{Nis}    ( X\times Y ,\cal F ) \to 
%H^i_{Nis}    ( X\times Y , j_*(j^*(\cal F) )\to 
%\cdots
%$$%$$\xymatrix{
%which is natural in respect to $Y$.
%Then applying the functor $\cal P\to \cal P(-\times\affl/0)$ we get the sequence
$$
\cdots \to H^{i-1}( Z\times\affl/0 , \cal F_{-1} ) \to 
H^i_{Nis}( X\times \affl/0 ,\cal F )\to 
H^i_{Nis}( U\times \affl/0 ,\cal F )\to 
\cdots
$$
The clam follows.

\end{proof}

\begin{proof}[Proof of the lemma \ref{lm:strinv-locbase}]
Using denotation \ref{def:F-1F(affl/0)} 
the claim is to prove that $H^i_{nis}( X\times \affl/0 , \cal F ) = 0$.
%We will prove 
Let's prove this 
%simultaneously for all homotopy invariant sheaves with GW-transfers
by the induction in respect to $i$. %??eng ord word
The base of the induction, i.e. the case $i=0$, is theorem \ref{th:HIAssZar}.
Suppose that the statement of theorem holds for all homotopy invariant presheaves for all $i<n$, for some $n$.

%Let $p\colon X\times\affl\to X$ denotes  projection
%and $i\colon X\to X\times\affl$ denotes embedding of zero section.
%The composition 
%  $i^*\circ p^*\colon H^i_{Nis}(X) \to H^i_{Nis}(X)$
%is identity.
%So to prove that  
%  $$p^*\colon  H^i_{Nis}(X)\to H^i_{Nis}(X\times\affl)$$ is isomorphism
%is equivalent to prove that  
%  kernel of $i^*$ is zero. 

Let $a\in H^i_{Nis}(X\times\affl/0,\cal F)$.
By lemma \ref{lm:strinv-pointbase} $H^i_{Nis}(\eta\times\affl/0)=0$, where $\eta$ denotes the generic point of $X$, and hence
  $a\big|_{U\times\affl/0}=0$ 
    for some open affine subscheme $U\subset X$.

Since $k$ is perfect 
the generic point $\eta_Z$ of subscheme $Z=X\setminus U$ is smooth and by lemma \ref{lm:strhominv:restrinj} we have 
$a_{\eta_Z\times \affl/0} = 0$.
Hence $a\big|_{U_1\times\affl/0}=0$ 
for some open subscheme $U_1\subset X$, $U_1\ni \eta_{Z}$
(we use here that $X$ is local).
Denote $Z_1=X\setminus U_1$, then $dim\, Z_1<dim\,Z$.
Now consequently applying lemma \ref{lm:strhominv:restrinj}
we can increase the dimension of the closed subscheme $Z_*$ to zero, %empty, 
so the claim follows.
%
%applied to $U_1$ and $U$ 
%  $a\big|_{U_1\times\affl}=0$.
%And by induction we can prove that 
%$$
%a\big|_{U_i\times\affl}=0,\;\text{where}\;
%U_i = X\setminus Z_i,\;Z_i = sing\, Z_{i-1}.$$
%Since $k$ is perfect
%  $$dim\, Z> dim\, Z_1 >\dots > dim\, Z_i >\dots .$$
%Hence for some finite $i$ 
%  $Z_i = \emptyset$ and $U_i = X$.  
%Thus we get that
%  $a\big|_{X\times\affl}=0$ for any $a\in ker i^*$.
\end{proof}

\begin{proof}[Proof of the theorem \ref{th:strhominv}]
Let $\cal F$ be a $\affl$-homotopy invariant presheave with GW-transfers.
Consider the Leray spectral sequence for the morphism of small Nisnevich sites $pr\colon (X\times \bGm)_{nis}\to X_{nis}$: 
$$
H^p_{Nis}(X,R^i{pr}(\widetilde{\cal F}\big|_{X\times\bGm})) \Rightarrow H^{p+q}_{Nis}(X\times\affl,\widetilde{\cal F}_{nis}) 
.$$
Lemma \ref{lm:strinv-locbase} implies
$R^i{pr}(\cal F\big|_{U\times\affl})(U_x)\simeq 0$ for $i>0$ and 
$R^0{pr}(\cal F\big|_{U\times\affl})(U_x) = {pr_*}(\cal F\big|_{U\times\affl})(U_x)=\cal F(U_x\times \affl)$.
So %this 
the spectral sequence degenerates and
$$
H^i( X, \widetilde{\cal F}_{nis} ) \simeq 
H^i( X, {pr_*}(\widetilde{\cal F}_{nis}\big|_{U\times\affl}) ) \simeq
H^i( X, \widetilde{\cal F}_{nis}\big|_{U} )
.$$
\end{proof}

\begin{remark}
Since there is a functor $GWCor\to WCor$ and proofs are based on explicit constructions of correspondences,
then this yields the similar results for presheaves with Witt-transfers.
\end{remark}

\begin{corollary}
%Homotopy invariant presheaves with GW-transfers and 
The Nisnevich cohomology presheaves of the associated sheave of a homotopy invariant presheaves with GW-transfers are representable in the motivic homotopy category $\mathbf H^{\affl}(k)$. 
\end{corollary}\begin{proof}
Consider the simplicial presheave $\mathcal I_{s}$ corresponding (via the Dold-Kan correspondence)
to the injective resolvent $\mathcal I^\bullet$ of the associated sheave $\widetilde{\cal F}_{Nis}$ 
for a homotopy invariant presheave with GW-transfers $\mathcal F$.
Then since $\mathcal I^\bullet$ is a bounded above complex of injective Nisnevich sheaves, it follows that  
$\mathcal I_{s}$ fulfils the Nisnevich-Mayer-Viertoris, and 
theorem \ref{th:strhominv} implies that $\mathcal I_{s}$ fulfils homotopy invariance.
Thus the claim follows from theorem 3.1 of \cite{Horn_ReprKOWitt}.
\end{proof}

\end{document}